\documentclass[12pt]{article}

\usepackage[inner=32mm,outer=31mm,tmargin=30mm,bmargin=40mm]{geometry}

\usepackage{latexsym,amsfonts,amsmath,amssymb,epsfig,tabularx,amsthm,dsfont,mathrsfs}

\usepackage{graphicx}
\usepackage{enumerate}

\usepackage{titling}
\newcommand{\subtitle}[1]{%
  \posttitle{%
    \par\end{center}
    \begin{center}\large#1\end{center}
    \vskip0.5em}%
}

\usepackage{hyperref} %[backref=page]
\hypersetup{colorlinks=true,
        linkcolor=black,
        citecolor=black,
        urlcolor=blue}

\theoremstyle{plain}

\newtheorem{thm}{Theorem}[section]
\newtheorem{lem}[thm]{Lemma}

\newtheorem{prop}[thm]{Proposition}
\newtheorem{cor}[thm]{Corollary}

\theoremstyle{definition}
\newtheorem{rem}[thm]{Remark}

\newtheorem{ass}[thm]{Assumption}
\newtheorem{defi}[thm]{Definition}
\newtheorem{alg}[thm]{Algorithm}

\renewcommand{\P}{{\mathbb P}}
\newcommand{\expect}{\operatorname{\mathbb{E}}}
\DeclareMathOperator{\Var}{\mathbb{V}ar}

\newcommand{\dint}{\,\mathup{d}}
\newcommand{\Finput}{\mathcal{F}}
\newcommand{\Yinput}{\mathcal{Y}}
\newcommand{\Barrel}{\mathcal{B}}
\newcommand{\ind}{\mathds{1}}

\DeclareMathOperator{\median}{med}

\newcommand{\wor}{\textup{wor}}
\newcommand{\fix}{\textup{fix}}

\DeclareMathOperator{\Int}{INT}

\newcommand{\eps}{\varepsilon}

\renewcommand{\rho}{\varrho}

\newcommand{\veci}{\mathbf{i}}

\newcommand{\vecp}{\mathbf{p}}

\newcommand{\vecx}{\mathbf{x}}

\newcommand{\vecz}{\mathbf{z}}
\newcommand{\vecX}{\mathbf{X}}

\newcommand{\Q}{{\mathbb Q}}
\newcommand{\R}{{\mathbb R}}

\newcommand{\N}{{\mathbb N}}

\newcommand{\D}{{\mathbb D}}

\DeclareMathAlphabet{\mathup}{OT1}{\familydefault}{m}{n}

\newcommand{\wt}{\widetilde}
\newcommand{\widebar}[1]{\mbox{\kern1.5pt\hbox{\vbox{\hrule height 0.6pt \kern0.35ex
        \hbox{\kern-0.15em \ensuremath{#1 }\kern0.0em}}}}\kern-0.1pt}

\newcommand{\fracts}[2]{{\textstyle\frac{#1}{#2}}}

\newlength{\fixboxwidth}
\usepackage{color}
\definecolor{darkgreen}{rgb}{0,0.5,0}
\definecolor{orange}{rgb}{1,0.4,0}

\newcommand{\checked}[1]{{\color{black} #1}}
\usepackage{soul}

\usepackage{latexsym,amsfonts,amsmath,amssymb,mathrsfs}
\usepackage{url}
\usepackage{graphicx}
\usepackage{color}
\usepackage{euscript}
\usepackage{verbatim}
\usepackage{hyperref}

\title{Solvable Integration Problems\\
and Optimal Sample Size Selection}

\date{\today}

\author{Robert J.\ Kunsch\thanks{Institut f\"ur Mathematik, 
Universit\"at Osnabr\"uck, Albrechtstr. 28a, 49076 Osnabr\"uck, Germany, 
Email: robert.kunsch@uni-osnabrueck.de},\; 
Erich Novak\thanks{Mathematisches Institut, Universit\"at Jena,
Ernst-Abbe-Platz 2, 07743 Jena, Germany,
Email: erich.novak@uni-jena.de},\;  
and Daniel Rudolf\thanks{Institute for Mathematical Stochastics, Universit\"at G\"ottingen 
\& Felix-Bernstein-Institute for Mathematical Statistics, 
Goldschmidtstra\ss e 7, 
37077 G\"ottingen, Germany, 
Email: daniel.rudolf@uni-goettingen.de}
}

\begin{document}
\maketitle

\begin{abstract}  
We compute the integral of a function
or the expectation of a random variable with minimal cost
and use,
for \checked{our new} algorithm and for upper bounds of the complexity, 
i.i.d.\ samples.
Under certain assumptions it is possible to select a sample size
\checked{based on} a variance estimation,
or -- more generally -- 
\checked{based on} an estimation of a (central absolute) \mbox{$p$-moment}. 
That way one can guarantee a small absolute error with high probability,
the problem is thus called solvable.
The expected cost of the method
depends on the $p$-moment of the random variable,
which can be arbitrarily large.

\checked{In order to prove the 
% Daniel: Wieso ``almost''?
%(almost) 
optimality of our algorithm we also 
provide lower bounds.} 
These bounds apply not only to methods based on i.i.d.\ samples but 
also to general randomized algorithms. 
They show that -- up to constants -- the cost of the algorithm
is optimal in terms of accuracy, confidence level,
and norm of the particular input random variable.
Since the considered classes of random variables or integrands
are very large,
the worst case cost would be infinite. 
Nevertheless one can define adaptive stopping rules such that 
%the average cost for every input is finite.
for each input the expected cost is finite.

We contrast these positive results with examples of integration problems
that are not solvable.
\end{abstract}
 
\section{Introduction} 

We want to compute the integral
\begin{equation}   \label{eq:INT} 
  \Int (f) = \int_G f(\vecx) \dint \pi(\vecx) 
\end{equation} 
of a $\pi$-integrable function $f: G \to \R$ %,
where $\pi$ is a probability measure on the domain~$G$,
or the expectation
\begin{equation}   \label{eq:E(Y)} 
  \expect Y = \int_\Omega Y(\omega) \dint \P(\omega) 
\end{equation} 
of a random variable $Y$ %,
mapping from a probability space $(\Omega,\Sigma,\P)$ to $\R$.
We allow randomized algorithms
and want to achieve this goal
up to some absolute error tolerance $\eps > 0$,
based on finitely many function values or i.i.d.\ samples, respectively.

Ideally, we would like to do it for an arbitrary 
$L_1$-function $f$ or an arbitrary integrable 
random variable $Y$. 
But there is no algorithm for this task, 
even if we assume %a
$C^\infty$-smoothness of the integrand, see Theorem~\ref{thm:Cinf}. 
Hence, we have to shrink the set of possible inputs. 
In this paper we define solvable integration problems in a very 
natural way: $\Int$ is \emph{solvable} for a class~$\Finput$ 
of integrands if, for every $\eps,\delta >0$, there is an algorithm~$A$
such that 
\begin{equation}  \label{eq: eps_delta_approx}
 \P\{|A(f) - \Int(f)| \leq \eps\} \geq1- \delta,
\end{equation}
for all $f\in\Finput$.
% \en{The cost of $A$, and the number of used function evaluations, 
% may depend on $f$ but is finite for every $f$.} 
\checked{The cost of $A$, that is, the number of used function evaluations
may depend on $f$ but is almost surely finite.}
If \eqref{eq: eps_delta_approx} holds for $\eps,\delta >0$,
then the method~$A$ 
is called \emph{$(\eps,\delta)$-approximating}.
Analogously we define solvability for~$\expect$
and classes~$\Yinput$ of random variables. 
With this concept we can allow quite large sets $\Finput$. 
Observe that we deviate from the 
existing literature on information-based complexity (IBC), 
see for example the books of Novak and Wo\'zniakowski~\cite{NW08,NW10,NW12}, 
where ``solvable'' means that we have a uniform cost bound and usually 
only bounded sets $\Finput$ are considered.

In this paper we mainly discuss input classes of the form 
\begin{equation}   \label{eq:cone}
  \begin{split}
  \Finput_{p,q,K} = \Finput_{p,q,K}(\pi)
  &:= \{ f \in L_q(\pi) \mid \Vert f - \Int (f) \Vert_q
    \leq K \, \Vert f - \Int (f) \Vert_p
    \} \,, \\
    \Yinput_{p,q,K}
    &:= \{ Y \in L_q \mid \| Y - \expect Y\|_q
    \leq K \, \|Y - \expect Y \|_p
    \} \,,
  \end{split}
\end{equation} 
where $1 \leq p < q \leq \infty$ and $K>1$.
The assumption, in fact, is a cone condition on (central) norms,
%  for $Y \in \Yinput_{p,q,K}$ any scaled and shifted version $aY + c$
%	 with $a,c \in \R$ is still in the input class.
% Robert: Vorher war die Sprache auf Funktionen,
%  				das sollte man vielleicht beibehalten?
% Daniel
for $f \in \Finput_{p,q,K}$ \checked{(resp.\ $Y \in \Yinput_{p,q,K}$)}
  any scaled and shifted version $af + c$ \checked{(resp.\ $aY + c$)}
  with $a,c \in \R$ is still in the input class.
These cone-shaped classes give a theoretical foundation to the idea
of adaptively determining a sample size for i.i.d.\ sampling
on ground of a variance estimation,
or -- more generally -- a $p$-moment estimation. 
The case $p=2$ and $q=4$ (bounded kurtosis)
was studied by Hickernell, Jiang, Liu, and Owen~\cite{HJLO13}. 

We stress that such problems cannot be solved in the sense
that the expected error (or root mean squared error)
is bounded for the whole class,
see Theorem~\ref{thm:E(err)<eps--noway}.
Hence, the understanding of ``solvable'' as
``small error with high probability''
is essential here.

% In Section~\ref{sec:UBs} we prove upper bounds
% for the
% classes~$\Yinput_{p,q,K}$ (and $\Finput_{p,q,K}$)
% \checked{by means of a new algorithm} that takes  i.i.d.\ samples.
\checked{We provide a new algorithm $A_{p,q,K}^{\eps,\delta}$
that solves \eqref{eq:E(Y)} %(resp. \eqref{eq:E(Y)})
on classes~$\Yinput_{p,q,K}$
(respectively, \eqref{eq:INT} on $\Finput_{p,q,K}$)
and prove an upper bound of the cost.
To measure the cost, we consider 
the random variable
$n(\omega,Y)$ 
that determines the sample size of the algorithm,
i.e.\ the number of copies of~$Y$ used by the algorithm.
Then the \emph{expected cost} for some input $Y$ is given by 
$\bar{n}(A_{p,q,K}^{\eps,\delta},Y):= \expect n(\,\cdot\,,Y)$.  
Our main results are:}
\begin{itemize}
 \item \checked{For arbitrary error thresholds~$\eps>0$
    and 
    \checked{uncertainty levels $\delta \in (0,1/2)$}
    we construct an $(\eps,\delta)$-approximating
    algorithm~$A_{p,q,K}^{\eps,\delta}$
    such that for~$Y \in \Yinput_{p,q,K}$ 
    %  the \emph{average cost}
    %  the \emph{expected cost}
    %  $\bar{n}(A_{p,q,K}^{\eps,\delta},Y):= \expect n(\,\cdot\,,Y)$ satisfies
    \begin{equation}           \label{eq:UBintro} 
      \bar{n}(A_{p,q,K}^{\eps,\delta},Y)
        \leq C_q \, K^{pq/(q-p)}
              \, \left(1 + \left(\frac{\|Y - \expect Y\|_1}{\eps}
              \right)^{\max\{1+1/(q-1),\,2\}}
              \right)
              \, \log \delta^{-1}
    \end{equation}
    where
    %$Y \in \Yinput_{p,q,K}$, and
    $C_q > 0$,
    see Theorem~\ref{thm:costbound} in Section~\ref{sec:UBs}.}
  \item \checked{In addition to that, in Section~3 we show matching lower bounds
    holding for \emph{all} algorithms, which we explain carefully in the following.}
\end{itemize}
% For arbitrary error thresholds~$\eps>0$
% and confidence levels \mbox{$1-\delta \in (0,1)$}
% we construct an $(\eps,\delta)$-approximating
% algorithm~$A_{p,q,K}^{\eps,\delta}$
% such that for~$Y \in \Yinput_{p,q,K}$ 
% %  the \emph{average cost}
% %  the \emph{expected cost}
% %  $\bar{n}(A_{p,q,K}^{\eps,\delta},Y):= \expect n(\,\cdot\,,Y)$ satisfies
% \begin{equation}           \label{eq:UBintro} 
%   \bar{n}(A_{p,q,K}^{\eps,\delta},Y)
%     \leq C_q \, K^{pq/(q-p)}
%           \, \left(1 + \left(\frac{\|Y - \expect Y\|_1}{\eps}
%           \right)^{\max\{1+1/(q-1),\,2\}}
%           \right)
%           \, \log \delta^{-1}
% \end{equation}
% where $Y \in \Yinput_{p,q,K}$,
% and $C_q > 0$, see Theorem~\ref{thm:costbound}.
%  Here, $n(\rjk{\omega},Y)$ 
%  is the random variable
%  that determines the sample size of the algorithm,
%  i.e.\ the number of copies of~$Y$ used by the algorithm.
% Matching lower bounds holding for \emph{all} algorithms
% are contained in Section~\ref{sec:LBs}.

\checked{Irrespective} of the specified accuracy~$\eps$
and the input random variable,
the cost bound~\eqref{eq:UBintro} is at least
some constant times $K^{pq/(q-p)} \, \log \delta^{-1}$.
This fixed cost comes from estimating $\|Y-\expect Y\|_1$
within our particular algorithm and
can be interpreted as the price we need to pay
for not knowing the statistical dispersion of~$Y$.
Theorem~\ref{thm:fixedcost} shows that such constant effort
is inevitable for all \mbox{$(\eps,\delta)$-approximating} 
algorithms,
with a lower bound that reproduces the dependence on~$K$ and $\delta$.
The cost bound~\eqref{eq:UBintro} also reveals
the adaption
of our algorithm
to the input,
the lower bounds of Theorem~\ref{thm:LB-YcapB2} and~\ref{thm:LB-YcapBq}
exhibit the same dependence on $\eps$, $\delta$,
and the norm.
In doing so, we obtain (up to constants) tight upper and lower bounds
for the $(\eps,\delta)$-complexity of computing expected 
values~\eqref{eq:E(Y)}.

For classes of inputs like~\eqref{eq:cone},
problems \eqref{eq:INT} and \eqref{eq:E(Y)}
are closely related.
However,
the class of (randomized) algorithms for the
integration problem~\eqref{eq:INT},
as considered in numerical analysis or complexity theory,
is usually broader than the class of algorithms
that are used in statistics for the
expectation problem~\eqref{eq:E(Y)}.
For the
% Robert: Mein Sprachgefuehl ist verwirrt von der Anforderung des Referees,
%         hier den Artikel zu entfernen.
%         Das kommt vielleicht daher, dass er die Referenz als Wort empfindet,
%         ich (und ihr anscheinend auch?) jedoch nur als zusaetzlichen Hinweis,
%         als Einschub, was wir mit ``numerical problem'' meinen.
%	  Weil zu oft auf die Gleichung referieren langweilig ist,
%	  schlage ich jetzt alternativ vor, INT und \expect zu nehmen.
%         Diese unsere Entscheidung wird wohl nicht die Revision gefaehrden...
% Daniel: Bzgl. dem Artikel: Naja, das koennt Ihre entscheiden, ich gehe auch sehr gern mit dem 
% 	Referee mit und der Artikel kann weg. Bzgl.: \Int versus (1), tendiere ich zu (1) 
% 	und Langeweile. Kann aber gut mit Int leben, deshalb habe
% 	ich es blau gemacht.
numerical problem~\checked{$\Int$} %~\eqref{eq:INT}
we consider all randomized algorithms  
that use finitely many function values for any input $f \in \Finput$,
whereas for the statistical problem~\checked{$\expect$} %~\eqref{eq:E(Y)}
the kind of available information is restricted: 
We may use only information from 
realizations of independent random variables~$Y_1, Y_2, \ldots, Y_n$,
which all have the same distribution as the 
\emph{basis experiment}~$Y \in \Yinput$,
where $n=n(\checked{\omega},Y)$ is an almost surely finite random variable
given by a stopping rule
that depends on the already observed experiments.
On the other hand, methods designed for computing expected values
can be used for integration problems by building Monte Carlo experiments
$Y_i:= f(\vecX_i)$ with
i.i.d.\ sampling nodes
\mbox{$\vecX_1,\vecX_2,\ldots \stackrel{\text{iid}}{\sim} \pi$}.
Usually,
the interpretation of computing expected values as a statistical problem
means that no deeper knowledge on the probability space on which~$Y$ is defined
is taken into consideration.
One can design algorithms that work for simulated samples
as well as real-live data, all of which are provided by the user.
In contrast,
the numerical interpretation opens the possibility of including knowledge
on the local structure of integrands~$f$ from the input class,
typically smoothness assumptions,
which could make simple i.i.d.\ sampling inferior to more sophisticated methods.
Such locality, however, is not present
in the definition of the classes~$\Finput_{p,q,K}$.
For the wider class of general randomized algorithms
that are feasible for the integration problem~\eqref{eq:INT},
lower bound proofs are more technical
and postponed to Section~\ref{sec:INTLBs}.
Actually, we reproduce the lower bound results
for i.i.d.-based algorithms,
see in particular Corollary~\ref{cor:LBsINT}.

%% Robert: Das sollte doch jetzt im Outlook stehen?
% Daniel: Finde ich nicht! Mir ist das recht hier, im Outlook darf das auch nochmal vorkommen. 
The new results for complexity on cones~\eqref{eq:cone} are different from
classical statements made in numerical analysis,
where the cost of an algorithm is given
by the cost for the worst input from a class of inputs
(usually a bounded set).
Here,
the cost is infinite for the whole class~$\Yinput_{p,q,K}$,
but it is finite for any single input~$Y$.	

\subsubsection*{Historical remarks.} 
Varying sample size has been considered in statistical testing since the 1940s,
so-called \emph{sequential tests} have been 
promoted by Wald 1945~\cite{Wa45} and 1948~\cite{Wa48}.
Later books on sequential methods include
Siegmund 1985~\cite[Chapter VII]{Sieg85}.
We also refer to the literature given in these sources.
Most of the early works deal with quite restricted assumptions
on the class of input random variables, e.g.\ normality in Stein 1945~\cite{St45}.
Other results under more general assumptions,
e.g.\ finite (but unknown) variance,
are of asymptotic nature, see Chow and Robbins 1965~\cite{CR65}.
To our knowledge, Hickernell et al.\ 2013~\cite{HJLO13}
were the first who by introducing the cones~$\Yinput_{2,4,K}$
provided general conditions on distributions
such that the mean of a random variable with unknown variance can be estimated
to within a given absolute error threshold with guaranteed confidence.
We improve the algorithm and the upper bound 
and also prove matching lower bounds.

Studies on optimally adapted sample size selection
can be found in papers of Dagum et al.~\cite{DKLR00},
Gajek, Niemiro, and Pokarowski~\cite{GNP13}, and Huber~\cite{Hu17Bernoulli}.
These authors study 
the complexity of approximating the expected value~$\expect Y > 0$
of a random variable~$0 \leq Y \leq 1$
to within a prescribed
relative error,
in detail \mbox{$|A(Y) - \expect Y|/\expect Y \leq \eps$},
with high confidence~$1-\delta$.
In such a situation,
the cost is roughly inversely proportional to the target value~$\expect Y$,
hence increasing as the latter approaches~$0$.
The stopping rule is rather simple:
The sample size~$n$ is chosen as the first number
such that~\mbox{$Y_1 + \ldots + Y_n \geq r$} for some given threshold~$r > 0$,
the output is then~$r/n$.
Here, \mbox{$r \asymp \eps^{-2} \log \delta^{-1}$}.

\checked{
\subsubsection*{Outlook.}

This paper is in the tradition of IBC, see \cite{No88,TWW88}.
A recent survey on the complexity of the integration problem is \cite{No16}.
Let $F_1$, $F_2$ and $R$ be normed linear spaces.
For a (linear) solution operator $S: F_2 \to R$,
in IBC one usually assumes 
that the input $f$ is from a ball $\Barrel_2 \subset F_2$. 
In the underlying work we assume that $f$ is from a cone $\Finput_{p,q,K}$. 
The idea to take ``cones, not balls'' was suggested and studied 
by Hickernell and his colleagues, see \cite{CDHHZ14} and \cite{HJLO13}. 
Interesting cones could be of the form
\begin{equation} \label{eq:generalcone}
  \Finput
    = \{ f \in F_2 \mid \|f\|_{F_2} \leq K \, \|f\|_{F_1} \} , 
\end{equation}
where $\|\cdot\|_{F_i}$ denotes the norm of $F_i$ for $i=1,2$, and
$F_2$ is embedded into $F_1$
with $\|\cdot\|_{F_1} \leq C \, \|\cdot\|_{F_2}$ and $K > 1/C > 0$.
For numerical problems of this type 
% Robert: gleiche Fragen bei relativem Fehler!
one can ask:
Is there a deterministic or randomized algorithm that terminates  
(with finite cost, depending on $f$) for every input $f \in \Finput$
with a 
(probabilistic)
$\eps$-approximation
with respect to the error measured in the norm $\|\cdot\|_R$ of the space~$R$? 
The worst case cost might be infinite for any succeeding algorithm,
still one could examine whether the cost that an algorithm~$A$ exhibits
  for an input~$f$
  is close to the minimal cost that any algorithm which succeeds on~$\Finput$
  would possess for this particular input~$f$.
  In our setting,
  we partially answer this question with the concept of the minimal fixed cost,
  see Theorem~\ref{thm:fixedcost},
  but there we lack a dependence on the particular input.
  By intersecting~$\Finput$ with a ball~$\Barrel \subset F_2$, though,
  we can study at least the classical worst case complexity on $\Finput \cap \Barrel$
  to get an idea on how the cost must depend on the size of the input,
  compare Theorems~\ref{thm:LB-YcapB2} and~\ref{thm:LB-YcapBq}.
%%Robert: Ich bin ich nicht happy mit dem, was da unten steht. Dann lieber was substantielles (siehe Einfuegung drueber), auch wenn wir es urspuenglich nicht drin haben wollten.
% Daniel: Fuer mich ist es okay so wie es jetzt ist.
%but one can ask all complexity questions if we intersect $\Finput$ with a 
%ball $\Barrel$ within $F_2$.
%The point here is, of course, that the algorithm should work for every 
%$f \in \Finput$ and should be almost optimal for all 
%$\Finput \cap \Barrel$. 
Such an analysis seems to be appropriate for algorithms on the cone~$\Finput$
that
%In order to construct good algorithms 
%for the cone $\Finput$ one may
use %algorithms for approximating
an estimator for $\|f\|_{F_1}$ with an error estimate 
proportional to $\|f\|_{F_2}$.
Actually, 
a main ingredient of our algorithm is the approximation
of the (centered) $L_p$-norm 
of a function or random variable
in~$F_2 = L_q$.
The problem of estimating the $L_p$-norm
  was recently studied by Heinrich~\cite{H18}
  for functions from %the unit ball of
  a Sobolev space $F_2$.
The cone assumption in~\eqref{eq:generalcone}
  then implies an upper bound for the stronger norm~$\|f\|_{F_2}$,
  and optimal algorithms for the problem~$S$ on the ball~$\Barrel_2 \subset F_2$
  can be scaled accordingly to conclude the new method.
}

\section{A moment-adapted algorithm and its analysis}  \label{s2}
\label{sec:UBs}

\subsection{The general structure of the algorithm}

The aim is to compute the expected value~$\expect Y$
for random variables~$Y$ from the cone~$\Yinput_{p,q,K}$, see~\eqref{eq:cone}.
%Here, we denote $\|X\|_p := (\expect |X|^p)^{1/p}$ for random variables~$X$.
The special case $p = 2$ and $q = 4$ has been introduced by
Hickernell et al.~2013~\cite{HJLO13}
together with an algorithm that solves the problem.
We modify their method, that way providing an algorithmic solution
to the whole parameter range, $1 \leq p < q \leq \infty$,
and obtain optimal cost rates
in~$\eps$ and~$\delta$.

Similarly to Hickernell et al.~\cite{HJLO13}, see also Stein~\cite{St45},
we consider a two-stage algorithm.
In the first step we estimate the central absolute $p$-moment of~$Y$,
based on which we choose the number of samples that we take for the second step
in order to compute a good approximation of the mean~\mbox{$\expect Y$}.
In both stages a median of independent estimators is used.

\begin{alg}[$A_{k,p,m}^{k',s,\eta}$] \label{alg:A_kpmksh}
  \emph{Parameters:\;} odd natural numbers $k$ and $k'$,
  integer $m \geq 2$;
  real numbers \mbox{$p,s \geq 1$}, and $\eta > 1$.\\
  \emph{Input:\;} Generator for i.i.d.\ copies~$Y_1,Y_2,\ldots$
    of a random variable~$Y$.\\
  \emph{Proceed as follows:\;}
  \begin{enumerate}
    \item Realize $k$~independent $p$-moment estimations
      using~$n_1 := k\,m$ observations,
      \begin{equation*}
        \widehat{R}_m^{(\ell)}
          := \frac{1}{m} \sum_{i = (\ell-1) m + 1}^{\ell m}
                |Y_i - \widehat{M}_m^{(\ell)}|^p \,,
        \qquad \text{where}\quad
        \widehat{M}_m^{(\ell)}
          := \frac{1}{m} \sum_{i = (\ell-1) m + 1}^{\ell m} Y_i \,,
      \end{equation*}
      for $\ell = 1,\ldots,k$,
      and determine their median (which is well defined since~$k$ is odd),
      \begin{equation*}
        \widehat{R}_{k,m} := \median(\widehat{R}_m^{(1)},\ldots,\widehat{R}_m^{(k)}) \,.
      \end{equation*}
    \item Realize $k'$ independent mean estimations,
      using $n_2 := k'\,m'$ additional observations,
      where we choose
      \mbox{$m' := \max\{\lceil \eta \, (\widehat{R}_{k,m})^s \rceil,\, 1\}$},
      \begin{equation*}
        \wt{M}_{m'}^{(\ell)}
          := \frac{1}{m'} \sum_{i = n_1 + (\ell-1) m' + 1}^{n_1 + \ell m'} Y_i \,,
      \end{equation*}
      here $\ell = 1,\ldots,k'$.
      Finally, output the  median (which is well defined since~$k'$ is odd),
      \begin{equation*}
        A_{k,p,m}^{k',s,\eta}(Y)
          = \wt{M}_{k',m'}
          := \median(\wt{M}_{m'}^{(1)},\ldots,\wt{M}_{m'}^{(k')}) \,.
      \end{equation*}
  \end{enumerate}
  The \checked{overall} cost is~%
  \mbox{$n(\omega,Y) = n_1 + n_2
            = k \, m
              + k' \,\max\{\lceil \eta \, (\widehat{R}_{k,m}(\omega))^s \rceil,\, 1\}$}.
\end{alg}

The main task for the subsequent sections is to determine
appropriate values for the parameters~$k$, $k'$, $m$, $s$, and $\eta$,
such that the algorithm~$A_{k,p,m}^{k',s,\eta}$ 
is $(\eps,\delta)$-approximating on the class $\Yinput_{p,q,K}$, 
that is,
\begin{equation*}
  \P\left\{| A_{k,p,m}^{k',s,\eta}(Y)-\expect Y |
    \leq \eps\right\}\geq 1-\delta
\end{equation*}
for all $Y\in\Yinput_{p,q,K}$ and $\eps,\delta>0$.
Hence the problem \eqref{eq:E(Y)} is solvable on $\Yinput_{p,q,K}$,
see Theorem~\ref{thm:solvability}.
The second task is then to control the expected cost of the algorithm
for any fixed input~$Y$,
and this is where~$k$ might need to be larger than~$k'$,
see Lemma~\ref{lem:E(R^s)}.
We are not confined to estimating the $p$-moment
for~\mbox{$Y \in \Yinput_{p,q,K}$},
indeed, any~$\wt{p} \in [1,q)$ is possible.
We provide all the tools to understand these variants
and briefly comment on possible strategies, see Remark~\ref{rem:balance}.
For simplicity, however, we focus on
an algorithm~$A_{\eps,\delta}^{p,q,K}$
that does the job via an estimate of the first central moment
(i.e.\ $\wt{p} = 1$),
though the cost bound 
involves probably suboptimal $K$-dependent constants
for the second stage then.

The described algorithm uses a method of `probability amplification',
namely the `median trick', 
see Alon et al.~\cite{AMS96} and Jerrum et al.~\cite{JVV86}.
The intention is to build an estimator~$\widehat{\theta}$
for an unknown quantity~$\theta \in \R$
such that with high probability the outcome lies within
a certain interval~$I \subset \R$.
%Usually, %% Robert: eigentlich immer, sonst sinnlos.
This interval depends on~$\theta$ and
represents the deviation of $\widehat{\theta}$ from~$\theta$
that we are willing to tolerate.
The following proposition is a minor modification of
\cite[Proposition~2.1, see also (2.6)]{NiPo09} from Niemiro and Pokarowski.

\begin{prop}[Median trick] \label{prop:med}
  Let~$(\widehat{\theta}_{\ell})_{\ell\in\N}$ be a sequence of real-valued i.i.d.\ 
  random variables,
  and let~\mbox{$I \subset \R$} be an interval such that
  \begin{equation*}
    \P\{\widehat{\theta}_{\ell} \in I\} \geq 1 - \alpha \,,
  \end{equation*}
  where~$0 < \alpha < 1/2$.
  For an odd natural number~$k$, define 
  the median of $k$ independent estimators,
  \mbox{$\widehat{\theta}
          := \median(\widehat{\theta}_1,\ldots,\widehat{\theta}_k)$}.
  Then
  \begin{equation*}
    \P\{\widehat{\theta} \in I\}
      \geq 1 - \frac{1}{2} (4\alpha(1-\alpha))^{k/2} \,.
  \end{equation*}
\end{prop}

\begin{rem}
 Already Nemirovsky and Yudin 1983 \cite[Proposition on p.~244]{NY83}
 suggest a similar probability amplification scheme.
 Recently, Devroye et al.~\cite{DLLO16} and Huber~\cite{Hu17relVar}
 worked on improvements of the median-of-means.
 Finally, let us mention that several proposals 
 for a generalization of the median-of-means
 for regression problems have been published, see 
\cite{HS16,JL16,JLO17,LO12,LM16,LM17,Mi15}.
\end{rem}

\subsection{Some interpolation results}

The condition %conditions
 that defines the cone~$\Yinput_{p,q,K}$ %have
\checked{has} implications
on the relation between other central absolute moments
for exponents~\mbox{$1 \leq \wt{p} < \wt{q} \leq q$}.

\begin{lem}\label{lem:RVprq}
  Let $Z \in L_q$ be a random variable
  and~\mbox{$1 \leq p < r < q \leq \infty$}
  where $\frac{1}{r} = \frac{1-\lambda}{p} + \frac{\lambda}{q}$
  for some~$0 < \lambda < 1$.
  Then
  \begin{equation*}
    \|Z\|_p
      \leq \|Z\|_r
      \leq \|Z\|_p^{1-\lambda} \|Z\|_q^{\lambda} \,.
  \end{equation*}
\end{lem}
\begin{proof}
  The first inequality is well-known as a consequence of Jensen's inequality.
  For the second inequality we use H\"older's inequality,
  without loss of generality \mbox{$Z \geq 0$}, and
  \begin{equation*}
    \|Z\|_r^r
      = \expect \left[Z^{(1-\lambda) r} \, Z^{\lambda r}\right]
      \leq \| Z^{(1-\lambda) r}\|_{\frac{p}{(1-\lambda)r}}
              \, \| Z^{\lambda r}\|_{\frac{q}{\lambda r}}
      = \|Z\|_p^{(1-\lambda) r} \|Z\|_q^{\lambda r} \,,
  \end{equation*}
  finishes the proof.
\end{proof}

As an immediate consequence we obtain the following inclusion properties:

\begin{cor}\label{cor:pq-cone_subset}
  Let \mbox{$1 \leq p < r < q \leq \infty$}
  where~\mbox{$\frac{1}{r} = \frac{1-\lambda}{p} + \frac{\lambda}{q}$}.
  Then
  \begin{enumerate}
    \item $\Yinput_{p,q,K}
              \subset \Yinput_{r,q,K}
              \subset \Yinput_{p,q,K'}$\;
      with\; $K' := K^{\frac{1}{1-\lambda}}
                  = K^{\frac{r\,(q-p)}{p\,(q-r)}}
                  > K$,
    \item $\Yinput_{p,q,K} \subset \Yinput_{p,r,K'}$\;
      with\; $K' := K^{\lambda}
                  = K^{\frac{q\,(r-p)}{r\,(q-p)}}
                  < K$.
  \end{enumerate}
  In particular,
  \begin{itemize}    
    \item
      we have
      \;\mbox{$\Yinput_{p,q,K} \subset \Yinput_{1,q,K'}$}
      \;with\; $K' := K^{\frac{p\,(q-1)}{q-p}}$\;
      \checked{(that is~$K' = K^{p}$ for $q = \infty$)},
    \item
      for~$q \geq 2$ we have
      \;\mbox{$\Yinput_{p,q,K} \subset \Yinput_{1,2,K'}$}
      \;with\; $K' := K^{\frac{p\,q}{2\,(q-p)}}$\;\\
      (that is~$K' = K^{p/2}$ in the case $q = \infty$).
  \end{itemize}
\end{cor}
\begin{proof}
  Apply Lemma~\ref{lem:RVprq} to $Z := Y - \expect Y$.
\end{proof}

\begin{rem}    
  A special case of the above results shows how the cone
  that has been considered in Hickernell et al.~\cite{HJLO13}
  is embedded in a larger cone of $L_2$-functions,
  we have $\Yinput_{2,4,K} \subset \Yinput_{1,2,K^2}$,
  such that it is sufficient
  to provide algorithms that work for the larger cone.
\end{rem} 

We use an extension of a special case of the Marcinkiewicz-Zygmund inequality, 
see \cite{ReLi01} and \cite[Theorem~2]{vBEs65}
as well as \cite[Sect. 2.2.8, Proposition~3]{No88}. 
We provide a proof of this inequality
by
using the same technique as in \cite[Proposition~5.4]{He94} and \cite{RuSc15}. 
This leads to slightly better constants than those in the literature. 

\begin{lem} \label{lem:meanest}
  Let $1 < q \leq 4$.
  Then for any mean zero random variable~$Z \in L_q$
  and independent copies
  $Z_1,\ldots,Z_m$ of $Z$, that is, $Z_i \sim Z$ for $i=1,\ldots,m$,
  we have
  \begin{equation*}
    \left\|\frac{1}{m} \sum_{i=1}^m Z_i\right\|_q
      \leq\begin{cases}
            2^{2/q - 1} \, m^{-(1-1/q)} \, \|Z\|_q
              \quad& \text{if $1 < q \leq 2,$}\\
            (2 \sqrt{3})^{1-2/q} \, m^{-1/2} \, \|Z\|_q
              \quad& \text{if $2 < q \leq 4$.}
          \end{cases}
  \end{equation*}
\end{lem}
\begin{proof}
  We apply the Riesz-Thorin interpolation theorem
  to an operator~$T P$, where
  \begin{equation*}
    P : L_q([0,1]) \to L_q([0,1]), \quad
      f \mapsto f - \Int (f)
  \end{equation*}
  projects onto the subspace of mean zero functions,
  \begin{equation*}
    L_q^0([0,1]) := \{f \in L_q([0,1]) \mid \Int (f) = 0\} \,,
  \end{equation*}
  and
  \begin{equation*}
    T : L_q([0,1]) \to L_q([0,1]^m), \quad
      f(x) \mapsto \frac{1}{m} \sum_{i=1}^m f(x_i)\,.
  \end{equation*}
  Note that random variables~\mbox{$Z \in L_q$}
  correspond to functions~\mbox{$f \in L_q([0,1])$} and
  mean zero random variables correspond to functions
  from the space~\mbox{$L_q^0([0,1])$}.
  We are interested in the quantity
  \begin{equation*}
  M_q := \sup_{\substack{\|f\|_q \leq 1\\
      \Int(f) = 0}}
      \|T f\|_q
      = \|T\|_{L_q^0 \to L_q}
      \leq \sup_{\|f\|_q \leq 1}
      \|T P f\|_q
      = \|T P\|_{L_q \to L_q} \,.
  \end{equation*}
  Let~\mbox{$1 \leq p < r < q \leq \infty$}
  with~\mbox{$\frac{1}{r} = \frac{1-\lambda}{p} + \frac{\lambda}{q}$},
  where~\mbox{$0 < \lambda < 1$}.
  Then the Riesz-Thorin theorem states
  \begin{equation} \label{eq:Riesz-Thorin}
    \|T P\|_{L_r \to L_r}
      \leq \|T P\|_{L_p \to L_p}^{1-\lambda} \|T P\|_{L_q \to L_q}^{\lambda} \,,
  \end{equation}
  see~\cite[Chapter~4: Corollary~1.8]{BeSh88} for details
  and a more general formulation.
  Knowing the special cases
  \begin{equation*}
    \|P\|_{L_q \to L_q}
      = \begin{cases}
          2 \quad& \text{if $q \in \{1,\infty\}$,} \\
          1 \quad& \text{if $q = 2$,}
        \end{cases}
  \end{equation*}
  we can apply the Riesz-Thorin theorem also to $P$
  (similarly to \eqref{eq:Riesz-Thorin})
  and obtain
  \begin{equation} \label{eq:|P|}
    \|P\|_{L_q \to L_q}
      \leq 2^{|2/q-1|} \qquad \text{for\; $1 \leq q \leq \infty$} \,.
  \end{equation}
  
  It is an immediate consequence of the triangle inequality that
  \begin{equation} \label{eq:M_1}
    M_1 = \|T\|_{L_1^0 \to L_1} = 1\,.
  \end{equation}
  There is a direct way to estimate~$M_q$ for even $q\in\N$.
  We switch back to random variables.
  Let $X,X_i$ for $i=1,\ldots,m$ be i.i.d.\ random variables
  uniformly distributed in $[0,1]$
  so that \mbox{$Z_i := f(X_i) \sim f(X)$}
  are i.i.d.\ mean zero random variables
  for~\mbox{$f \in L_q^0([0,1])$}.
  Then
  \begin{equation} \label{eq:E|meanZ|^q,even}
    \expect\left|\frac{1}{m} \sum_{i=1}^m Z_i\right|^q
      = m^{-q} \sum_{\veci \in \{1,\ldots,m\}^q} \expect \prod_{k=1}^q Z_{i_k} \,.
  \end{equation}
  Observe
  that \mbox{$\expect \prod_{k=1}^q Z_{i_k}$} vanishes
  if there is at least one~\mbox{$i \in \{1,\ldots,m\}$}
  that occurs only once in~\mbox{$\veci = (i_1,\ldots,i_q)$}.
  (This is due to independence of the~$Z_i$ and~\mbox{$\expect Z_i = 0$}.)
  For all the other terms a generalized version of H\"older's inequality
  yields
  \begin{equation*}
    \expect \prod_{k=1}^q Z_{i_k} \leq \|Z\|_q^q \,.
  \end{equation*}
  It remains to count the terms that do not vanish, for 
  the two relevant special cases we obtain:
  \begin{description}
    \item[\normalfont{\mbox{$q=2$:}}]
      There are $m$ terms of the shape~$Z_i^2$.
    \item[\normalfont{\mbox{$q=4$:}}]  There are $m$ terms of the shape~$Z_i^4$,
      and \mbox{$3 \, m(m-1)$}~terms of the shape~\mbox{$Z_i^2 Z_j^2$}
      with distinct indices~\mbox{$i \not= j$}.
      Altogether these are less than \mbox{$3 \, m^2$}~terms.
  \end{description}
  Plugging this into~\eqref{eq:E|meanZ|^q,even}
  and taking the $q$-th root we have
  \begin{equation} \label{eq:M_2,4}
    M_2 = m^{-1/2} \,,
    \qquad \text{and} \qquad
    M_4 \leq 3^{1/4} \, m^{-1/2} \,.
  \end{equation}
  
  We interpolate the operator norm of $T P$ between the special values
  for which by~\eqref{eq:|P|} together with \eqref{eq:M_1} and \eqref{eq:M_2,4}
  we have
  \begin{equation*}
    \|T P\|_{L_1 \to L_1} = 2 \,, \quad
    \|T P\|_{L_2 \to L_2} = m^{-1/2} \,, \quad \text{and}\quad
    \|T P\|_{L_4 \to L_4} \leq \sqrt{2} \, 3^{1/4} \, m^{-1/2} \,.
  \end{equation*}
  Hence, by the Riesz-Thorin~theorem as stated in \eqref{eq:Riesz-Thorin},
  for~\mbox{$p = 1 < r < q = 2$}
  with \mbox{$\lambda = 1- 2/r$},
  and for~\mbox{$p = 2 < r < q = 4$} with~\mbox{$\lambda = 2 - 4/r$},
  we obtain the desired upper bounds on~$M_q$ for $1 < q \leq 4$.
\end{proof}

\subsection{Estimating central absolute moments}

We consider estimators for the (central absolute)
$p$-moment~\mbox{$\rho^p := \expect |Y - \expect Y|^p$}
of a random variable~$Y \in \Yinput_{p,q,K}$,
based on $m$~independent copies~$Y_1,\ldots,Y_m$ of $Y$
and their residuals,
\begin{equation} \label{eq:R_m}
  \widehat{R}_m
    := \frac{1}{m} \sum_{i = 1}^m
          |Y_i - \widehat{M}_m|^p \,,
  \qquad \text{where}\quad
  \widehat{M}_m
    := \frac{1}{m} \sum_{i = 1}^{m} Y_i \,.
\end{equation}
The first moment and the corresponding estimator, that is for $p=1$,
is also known as `mean absolute deviation', `average absolute deviation',
or `mean deviation from the mean'.
This and other \emph{measures of scale}
like `median absolute deviation from the median'
constitute alternatives to the widely known standard deviation
as an empirical quantification of the dispersion of a random variable.

\begin{lem} \label{lem:R_m-error}
  Let $Y \in \Yinput_{p,q,K}$ and~$\rho := \|Y-\expect Y\|_p$.
  \begin{enumerate}[(a)]
    \item \label{en: a_R_m-error}
      If $1 = p < q \leq 2$,
      then
      \begin{equation*}
        \|\widehat{R}_m - \rho\|_q
          \,\leq\, \frac{2^{2/q-1} \, (1+2K)}{m^{1-1/q}} \, \rho
          \,\leq\, \frac{6 \, K}{m^{1-1/q}} \, \rho \,.
      \end{equation*}
    \item \label{en: b_R_m-error}
      Let $1 \leq p \leq 2$, and $p < q \leq 2 p$.
      Then, with \mbox{$r := q/p$}, we have
      \begin{equation*}
        \|\widehat{R}_m - \rho^p\|_r \leq \frac{26 \, K^p}{m^{1-1/r}} \, \rho^p \,.
      \end{equation*}
  \end{enumerate}
\end{lem}
\begin{proof}
  To shorten the notation, define $a := \expect Y$.
  
  We start with proving \eqref{en: a_R_m-error}.\;
%  From the triangle inequality it follows that
  \checked{Via the triangle inequality,
    we separate the inaccuracy of the empirical mean $\widehat{M}_m$
    from the deviation of a single observation~$Y_i$,}
  \begin{equation*}
    |Y_i - a| - |\widehat{M}_m - a|
      \,\leq\, |Y_i - \widehat{M}_m|
      \,\leq\, |Y_i - a| + |\widehat{M}_m - a| \,.
  \end{equation*}
  Hence, for the estimator~\eqref{eq:R_m} we can write
  \begin{equation*}
    \widehat{R}_m
      = \underbrace{\frac{1}{m} \sum_{i=1}^m |Y_i - a|}_{=: \widehat{\rho}_m}
          + r_m
  \end{equation*}
  %where
  \checked{with bounded remainder}~$|r_m| \leq |\widehat{M}_m - a|$.
%  With the triangle inequality we obtain
  \checked{For the $L_q$-norm this gives}
  \begin{align*}
    \|\widehat{R}_m - \rho\|_q
      &\,\leq\, \|\widehat{\rho}_m - \rho\|_q + \|r_m\|_q
      \,\leq\, \|\widehat{\rho}_m - \rho\|_q + \|\widehat{M}_m - a\|_q \,.
   \intertext{%
%        \end{align*}
   By Lemma~\ref{lem:meanest} we conclude}
%        \begin{align*}
    \|\widehat{R}_m - \rho\|_q
      &\,\leq\, \frac{2^{2/q - 1}}{m^{1-1/q}}
      \, \left(\||Y-a|-\rho\|_q + \|Y-a\|_q\right) \\
      &\,\leq\, \frac{2^{2/q - 1}}{m^{1-1/q}} \, \left(\rho + 2 \, \|Y-a\|_q\right)
      \,\leq\, \frac{2^{2/q - 1} \, (1 + 2K)}{m^{1-1/q}} \, \rho \,.
  \end{align*}
  This finishes the proof of \eqref{en: a_R_m-error}. 

  We turn to \eqref{en: b_R_m-error}.
  First note that, for $x,y \in \R$ and $p > 1$, we have
  \begin{equation} \label{eq:|x-y|^p}
    |x|^p - p \, |x|^{p-1} |y|
      \,\leq\, |x - y|^p
      \,\leq\, |x|^p + p \, |2x|^{p-1} |y| + |2y|^p \,,
  \end{equation}
        where the second inequality can be proven by distinguishing 
        the cases $|y| < |x|$ and $|y| \leq |x|$ with subcases
        $| x |\leq | x-y |$ and
        $| x | >| x-y |$.
%         
%         Daniel: Dem Leser ueberlassen.
%         
%   Indeed, concerning the second inequality,
%   the case~$|y| \leq |x|$ is trapped by the first two terms,
%   where we use that the Lipschitz constant of the function~\mbox{$y \mapsto |x-y|^p$}
%   is bounded by~\mbox{$p \, |2x|^{p-1}$} within that range.
%   The case~\mbox{$|y| \geq |x|$} is trapped solely by the third term.
%   The first inequality is somewhat easier.
  We use \eqref{eq:|x-y|^p} with~\mbox{$x = Y_i - a$} and \mbox{$y = \widehat{M}_m - a$}.
  Hence
  \begin{equation*}
    \widehat{R}_m
      \,=\, \underbrace{\frac{1}{m} \sum_{i=1}^m |Y_i - a|^p
                      }_{=:\widehat{\rho}_m^p}
              + r_m \,,
  \end{equation*}
   with \checked{the remainder~$r_m$ bounded by}
  \begin{equation*}
    |r_m|
      \,\leq\, \frac{2^{p-1} \, p}{m}
                  \sum_{i=1}^m
                    |Y_i - a|^{p-1} \, |\widehat{M}_m - a|
                  + 2^p \, |\widehat{M}_m - a|^p \,.
  \end{equation*}
  By the triangle inequality we obtain \checked{the norm estimate}
  \begin{multline*}
    \|\widehat{R}_m - \rho^p\|_r
      \,\leq\, \|\widehat{\rho}_m^p - \rho^p\|_r + \|r_m\|_r \\
      \,\leq\, \|\widehat{\rho}_m^p - \rho^p\|_r
                  + \frac{2^{p-1} \, p}{m}
                      \sum_{i=1}^m
                        \| |Y_i - a|^{p-1} \, |\widehat{M}_m - a| \|_r
                  + 2^p \, \| |\widehat{M}_m - a|^p \|_r \,.
  \end{multline*}
  \checked{We aim to control the norm of the product term.}
  For $L_q$-integrable non-negative random variables~$X,Z$,
  from H\"older's inequality, and recalling $pr = q$, we have
  \begin{equation*}
    \expect X^{r(p-1)} \, Z^r
      \leq \|X^{r(p-1)}\|_{p/(p-1)} \, \|Z^r\|_p
      \leq \|X\|_q^{r (p - 1)} \, \|Z\|_q^{r} \,.
  \end{equation*}
  \checked{Taking} \mbox{$X = |Y_i - a|$} and \mbox{$Z = |\widehat{M}_m - a|$},
  this leads to
  \begin{align*}
    \|\widehat{R}_m - \rho^p\|_r
      &\,\leq\, \|\widehat{\rho}_m^p - \rho^p\|_r
      + 2^{p-1} \, p \, \|Y - a\|_q^{p-1} \, \|\widehat{M}_m - a\|_q
      + 2^p \, \| \widehat{M}_m - a \|_q^p \,. 
  \end{align*}
  Finally, Lemma~\ref{lem:meanest} gives
  \begin{align*}
    &\|\widehat{R}_m - \rho^p\|_r \\
          &\,\leq\, 2^{2/r - 1} \, m^{-(1-1/r)} \, \| |Y - a|^p - \rho^p \|_r\\
          &\qquad    + \|Y - a\|_q^p
                %\left.
        \begin{cases}
        2^{p + 2/q - 2} \, p \, m^{-(1-1/q)} 
          + 2^{2/r} \, m^{-p(1-1/q)}
          \quad& \text{if $q \leq 2$,} \\
        2^{p - 1} \, p \, (2 \sqrt{3})^{1-2/q} \, m^{-1/2} \, 
        + 2^p \, (2 \sqrt{3})^{p(1-2/q)} \, m^{-p/2}
        \quad& \text{if $2 < q \leq 4$,}
        \end{cases}
        %\right\} 
        \\
    &\,\leq\, (2 \, \rho^p + 24 \, \|Y - a\|_q^p) \, m^{-(1-1/r)} 
    \,\leq\, \frac{26 \, K^p}{m^{1-1/r}} \, \rho^p \,,
  \end{align*}
  which finishes the proof.
\end{proof}

\checked{Combining} %Using
Markov's inequality
\checked{with the above result},
we obtain %the following
\checked{probabilistic bounds for the error of the empirical $p$-moment estimator}. 

\begin{lem} \label{lem:|R_m-rho|<}
  Let $Y \in \Yinput_{p,q,K}$
  with~$\rho := \|Y-\expect Y\|_p$,
  and let $0 < \alpha < 1$.
  \begin{enumerate}[(a)]
    \item If $1 = p < q \leq 2$, then
%     for the estimator~\eqref{eq:R_m} we have
      \begin{equation*}
        \P\left\{|\widehat{R}_m - \rho|
                    \leq \frac{2^{2/q - 1} \, (1+2K)}{\alpha^{1/q} \, m^{1-1/q}} \, \rho
          \right\}
            \,\geq\, 1 - \alpha \,.
      \end{equation*}
    \item In the case $1 \leq p \leq 2$, and $p < q \leq 2p$,
%       for the estimator~\eqref{eq:R_m} 
      we have
      \begin{equation*}
        \P\left\{|\widehat{R}_m - \rho^p|
                    \leq \frac{26 \, K^p}{\alpha^{p/q} \, m^{1-p/q}} \, \rho^p
          \right\}
            \,\geq\, 1 - \alpha \,.
      \end{equation*}
  \end{enumerate}
\end{lem}

From this, via Proposition~\ref{prop:med},
we derive the following probabilistic \checked{guarantees}
for the modified median-of-empirical-moments estimator~$\widehat{R}_{k,m}$
from Algorithm~\ref{alg:A_kpmksh},
which exhibit enhanced confidence levels.

\begin{prop} \label{prop:|R_km-rho|<}
  Let $Y \in \Yinput_{p,q,K}$
  with~$\rho := \|Y-\expect Y\|_p$,
  and let $0 < \alpha < 1/2$.
  \begin{enumerate}[(a)]
    \item If $1 = p < q \leq 2$,
      then for the first moment estimator from Algorithm~\ref{alg:A_kpmksh} we have
      \begin{equation*}
        \P\left\{|\widehat{R}_{k,m} - \rho|
                    \leq \frac{2^{2/q - 1} \, (1+2K)}{\alpha^{1/q} \, m^{1-1/q}} \, \rho
          \right\}
            \,\geq\,  1 - \frac{1}{2} (4 \alpha(1-\alpha))^{k/2} \,.
      \end{equation*}
    \item If $1 \leq p \leq 2$ and $p < q \leq 2p$, then
      for the $p$-moment estimator from Algorithm~\ref{alg:A_kpmksh} we have
      \begin{equation*}
        \P\left\{|\widehat{R}_{k,m} - \rho^p|
                    \leq \frac{26 \, K^p}{\alpha^{p/q} \, m^{1-p/q}} \, \rho^p
          \right\}
            \,\geq\,  1 - \frac{1}{2} (4 \alpha(1-\alpha))^{k/2} \,.
      \end{equation*}
  \end{enumerate}
\end{prop}

\begin{rem}[Choice of the parameters for the first stage]
  \label{rem:mk-choice}
  We need to control
  how badly we underestimate the $p$-moment of~$Y \in \Yinput_{p,q,K}$
  by the estimator~$\widehat{R}_{k,m}$. We consider either $1=p<q\leq2$ or $1\leq p\leq 2$
  and $p<q\leq 2p$.
  For an overall confidence level $1-\delta$ of the error of the algorithm
  we need
  a $p$-moment estimate that satisfies
  \begin{equation} \label{eq:R_km>wish}
    \P\left\{\widehat{R}_{k,m} \geq (1-\gamma) \, \rho^p \right\}
      \geq 1 - \delta_1
  \end{equation}
  for some~$0 < \gamma < 1$ and $0 < \delta_1 < \delta$.
  This situation can be established
  by the use of Proposition~\ref{prop:|R_km-rho|<}:
  \begin{itemize}
    \item Fix some \mbox{$0 < \alpha_1 < 1/2$}.
    \item Choose an odd
      \mbox{$k \geq 2 \log (2\delta_1)^{-1}
                        / \log (4 \alpha_1 (1-\alpha_1))^{-1}$},
      in order to guarantee the correct confidence level.
      For \mbox{$\alpha_1 = 1/4$} and~\mbox{$\delta_1 = \delta/2$},
      this simplifies to
      \mbox{$k \geq 2 \log \delta^{-1} / \log \frac{4}{3}$}.
    \item Choose 
      \begin{equation*}
        m \geq \begin{cases}
                  \displaystyle
                  \frac{2^{1/(q-1) - 1} \, (1 + 2K)^{1 + 1/(q-1)}
                        }{\gamma^{1 + 1/(q-1)}  \, \alpha_1^{1/(q-1)}}
                    &\quad \text{if $p=1 < q \leq 2$,}
                    \vspace{5pt}\\
                  \displaystyle
                  \frac{26^{1 + p/(q-p)} \, K^{pq/(q-p)}
                        }{\gamma^{1 + p/(q-p)} \, \alpha_1^{p/(q-p)}}
                    &\quad \text{if $1 < p \leq 2$ and $p < q \leq 2p$,}
                \end{cases}
      \end{equation*}
      in order to describe the correct event.
      With~$\alpha_1 = 1/4$ and $\gamma = 1/2$,
      and with further simplifications,
      it is sufficient to choose
      \begin{equation*}
        m \geq \begin{cases}
                  3 \cdot 48^{1/(q-1)} \, K^{1 + 1/(q-1)}
                    &\quad \text{if $p=1 < q \leq 2$,} \\
                  52 \cdot 208^{p/(q-p)} \, K^{pq/(q-p)}
                    &\quad \text{if $1 < p \leq 2$ and $p < q \leq 2p$.}
                \end{cases}
      \end{equation*}
  \end{itemize}
  The cost for estimating the $p$-moment is determined
  by the product~\mbox{$n_1 = k \, m$}.
  One can optimize the choice of $\alpha_1$ numerically in order to minimize this product.
  Such optimization depends on the ratio $q/p$.
  The optimal $\alpha_1$ approaches~$1/2$ as the ratio $q/p$ approaches~$1$.
  However, the choice $\alpha_1 = 1/4$ is sufficient.
  The parameters~$\gamma$ and $\delta_1$ affect the relative precision
  and the uncertainty at which we estimate the $p$-moment.
  This has consequences for the sample size needed in the second stage
  of the algorithm, see Section~\ref{sec:n_2}.
\end{rem}

\subsection{Determining the sample size}
\label{sec:n_2}

Given $1 \leq p < q \leq \infty$, we put $\wt{q} := \min\{q,2\}$.
For $Y \in \Yinput_{p,q,K}$
we define \mbox{$\tau := \|Y - \expect Y\|_{\wt{q}}$},
which is the norm of interest when determining the required sample size
for estimating the mean.
We start with considerations for mean estimators
assuming that~$\tau$ is known.
Then, fixing $m' \in \N$, for
\begin{equation*}
  \widehat{M}_{m'}
    := \frac{1}{m'} \sum_{i = 1}^{m'} Y_i \,,
    \qquad Y_i \stackrel{\text{iid}}{\sim} Y ,
\end{equation*}
Markov's inequality and Lemma~\ref{lem:meanest}
immediately imply the following.
\begin{lem} \label{lem:meandev}
  Let $Y \in L_{\wt{q}}$ 
  with $1 < {\wt{q}} \leq 2$ and $\tau := \|Y - \expect Y\|_{\wt{q}}$.
  Then we have for $0 < \alpha < 1$ that
  \begin{equation*}
    \P\left\{ |\widehat{M}_{m'} - \expect Y| \leq \frac{2^{2/{\wt{q}} - 1} \, \tau
                  }{\alpha^{1/{\wt{q}}} \, (m')^{1-1/{\wt{q}}}}
      \right\}
        \geq 1 - \alpha \,.
  \end{equation*}
\end{lem}
As a consequence, for the median-of-means estimator
\begin{equation*}
  \widehat{M}_{k',m'}
    := \median(\widehat{M}_{m'}^{(1)},\ldots,\widehat{M}_{m'}^{(k')})
  \qquad\text{where}\quad
  \widehat{M}_{m'}^{(\ell)}
    := \frac{1}{m'} \sum_{i = (\ell-1) m' + 1}^{\ell m'} Y_i  \,,
\end{equation*}
with fixed odd number $k'$,
via Proposition~\ref{prop:med} we obtain 
probabilistic \checked{error} bounds with enhanced confidence level.
A similar result for the median-of-means estimator
for $Y \in L_q$
with $1<q<2$
is obtained by Bubeck et al.\ in \cite[Lemma~2]{BC-BL13}.
The constants we get turn out to be slightly better.
\begin{prop} \label{prop:M_km-prob}
Let $Y \in L_{\wt{q}}$ with $1<\wt{q}\leq 2$ and 
$\tau := \|Y - \expect Y\|_{\wt{q}}$.
Then we have for $0 < \alpha < 1/2$ that
  \begin{equation*}
    \P\left\{ |\widehat{M}_{k',m'} - \expect Y| \leq \frac{2^{2/\wt{q} - 1} \, \tau
                }{\alpha^{1/\wt{q}} \, (m')^{1-1/\wt{q}}}
      \right\}
        \geq 1 - \frac{1}{2} (4 \alpha(1-\alpha))^{k'/2} \,.
  \end{equation*}
\end{prop}
If $\tau$ is known, then any choice
\begin{equation} \label{eq:m'(tau)}
  m' \geq \frac{2^{1/(\wt{q}-1)-1}}{\alpha^{1/(\wt{q}-1)}}
                      \, \left(\frac{\tau}{\eps}\right)^{1 + 1/(\wt{q}-1)}
\end{equation}
would give probabilistic bounds
for an absolute error~$\eps > 0$,
where the confidence level is determined by~$\alpha$ and $k'$.
Unfortunately, we do not know $\tau$
but rely on estimates of the central $L_p$-norm
\mbox{$\varrho = \|Y - \expect Y\|_p$} for $Y \in \Yinput_{p,q,K}$.
In the case $p = \wt{q} = 2$,
this is directly
the norm of interest $\varrho = \tau$.
If however \mbox{$1 \leq p < 2$},
we exploit the cone condition
\begin{equation*} % \label{eq: K'}
  \tau = \|Y - \expect Y\|_{\wt{q}}
       \leq K' \|Y - \expect Y\|_{p}
       = K' \rho
  \quad\text{where}\;
  K' := \begin{cases}
          K \quad&\text{if $q \leq 2$,} \\
          K^{\frac{q(2-p)}{2(q-p)}} \quad&\text{if $q > 2$,}
        \end{cases}
\end{equation*}
see Corollary~\ref{cor:pq-cone_subset}
for the inclusion~\mbox{$\Yinput_{p,q,K} \subset \Yinput_{p,2,K'}$}
in the case $q > 2 = \wt{q}$.
Remark~\ref{rem:mk-choice} provides parameter settings
for the $p$-moment estimator~$\widehat{R}_{k,m}$
%of Algorithm~\ref{alg:A_kpmksh}
such that
\begin{equation*}
 \P\Big\{ \rho^p \leq \widehat R_{k,m}/(1-\gamma) \Big\}
   \geq 1-\delta_1 \,.
\end{equation*}
If for the median-of-means estimator~\mbox{$\wt{M}_{k',m'}$}
of Algorithm~\ref{alg:A_kpmksh} we have that
\begin{equation} \label{eq:mean<wish}
  \P\left(|\wt{M}_{k',m'} - \expect Y|
            \leq \eps
          \,\middle|\, \rho^p \leq \widehat{R}_{k,m}/(1-\gamma)
    \right)
      \,\geq\, 1 - \delta_2 \,,
\end{equation}
with $\delta_1+\delta_2=\delta$, then
\begin{equation*}
  \P\{ | A_{k,p,m}^{k',s,\eta}(Y)-\expect Y | \leq \eps \}
    \,\geq\, (1 - \delta_1)(1 - \delta_2)
    \,\geq\, 1 - (\delta_1 + \delta_2)
    \,=\, 1 - \delta \,.
\end{equation*}
%% Einschub
  Hoping that \mbox{$\rho^p \leq \widehat{R}_{k,m}/(1-\gamma)$}
  holds indeed,
  the parameter~$m'$
  in %the estimator
  $\widetilde{M}_{k',m'}$ %from Algorithm~\ref{alg:A_kpmksh}
  can be chosen as in~\eqref{eq:m'(tau)},
  just replacing $\tau$ by
  \begin{equation} \label{eq:>tau}
    K'\cdot(\widehat{R}_{k,m}/(1-\gamma))^{1/p} \geq \tau \,.
  \end{equation}
  Hence, $m'$ is not a fixed number anymore
  but depends on the $p$-moment estimate.
  Still, since the conditional distribution of~$\wt{M}_{k',m'}$
  after the $p$-moment estimation is equal to the distribution of 
  $\widehat{M}_{k',m'}$ with the respective fixed $m'$,
  Proposition~\ref{prop:M_km-prob} may be used for the conditional case
  in order to obtain~\eqref{eq:mean<wish}.
  The consequential choice of the parameters for the second stage
  of Algorithm~\ref{alg:A_kpmksh} is summarized in the following remark.

\begin{rem}[Choice of the parameters for the second stage]
  \label{rem:eta-choice}
  Let $Y \in \Yinput_{p,q,K}$,
  with~\mbox{$1 \leq p \leq 2$} and \mbox{$p < q \leq 2p$}.
  From Remark~\ref{rem:mk-choice} we know how to choose $m$ and $k$
  within the first stage of $A_{k,p,m}^{k',s,\eta}$,
  see Algorithm~\ref{alg:A_kpmksh}. %, which gives us $\widehat{R}_{k,m}$.
  Now, with Proposition~\ref{prop:M_km-prob}
  and the considerations~\eqref{eq:m'(tau)}--\eqref{eq:>tau},
  we can choose $s,k'$ and $\eta$ depending on
  $\eps>0$, uncertainty $0<\delta_2<\delta$ and $p,q$ as follows:
  \begin{itemize}
    \item Fix some $0 < \alpha_2 < 1/2$.
    \item Choose an odd number
      \mbox{$k' \geq 2 \log(2\delta_2)^{-1}
      /\log(4\alpha_2 (1-\alpha_2))^{-1}$},
      in order to guarantee the conditional confidence level.
      For~\mbox{$\alpha_2 = 1/4$} and \mbox{$\delta_2 = \delta/2$},
      this 
      simplifies to
      \mbox{$k' \geq 2 \log \delta^{-1} / \log \frac{4}{3}$}.
    \item Choose $m'$
    according to \eqref{eq:m'(tau)} and \eqref{eq:>tau}
     with $\alpha_2$ instead of $\alpha$.
      For~\mbox{$\alpha_2 = 1/4$} and \mbox{$\gamma = 1/2$},
      it is sufficient to choose 
      \begin{align*}
        s &:= \fracts{1}{p} \max\left\{1 + \fracts{1}{q-1}, 2\right\}\\
        \eta &:=\begin{cases}
                  16^{1/(q-1)} \, K^{1 + 1/(q-1)} \, \eps^{-(1 + 1/(q-1))}
                  \quad&\text{for $q \leq 2$,} \\
                  \displaystyle
                  16 \, K^{q(2-p)/(q-p)} \, \eps^{-2}
                  \quad&\text{for $2 \leq q$.}
                \end{cases}
      \end{align*}
  \end{itemize}
\end{rem}

This gives algorithms $A_{k,p,m}^{k',s,\eta}$ 
for approximating \checked{the mean} $\expect Y$ %for
\checked{of random variables} $Y \in \Yinput_{p,q,K}$
%for
\checked{from cones with} very specific parameter ranges,
namely \mbox{$1 \leq p \leq 2$} and \mbox{$p < q \leq 2p$}.
These methods are also suitable for general cones~$\Yinput_{p,q,K}$
since by Corollary~\ref{cor:pq-cone_subset}
they are contained in specific cones,
in particular $\Yinput_{p,q,K} \subset \Yinput_{1,\wt{q},K'}$
with~\mbox{$\wt{q}:= \min\{q,2\}$} and some~$K'$.
In other words, it is always feasible to rely on a first moment estimate.
The following theorem summarizes the solvability result:
\begin{thm} \label{thm:solvability}
  The problem of computing expected values~\eqref{eq:E(Y)}
  is solvable for all input classes~$\Yinput_{p,q,K}$,
  \mbox{$1 \leq p < q \leq \infty$}, see~\eqref{eq:cone}.
  In detail, for any~\mbox{$\eps, \delta > 0$},
  the algorithm~$A_{k,1,m}^{k',s,\eta}$ (see Algorithm~\ref{alg:A_kpmksh}),
  with odd natural numbers
  $k,k' \geq 2 \frac{\log \delta^{-1}}{\log(4/3)}$,
  and
  \begin{itemize}
    \item for $1 \leq p < q \leq 2$ with
%     if $1 \leq p < q \leq 2$, put
      \begin{align*}
        m &:= \lceil 3 \cdot 48^{1/(q-1)} \, K^{p q/(q-p)} \rceil \,,\\
        s &:= 1 + \fracts{1}{q-1}\,,\\
        \eta &:= 16^{1/(q-1)} \, K^{p q/(q-p)} \, \eps^{-(1+1/(q-1))} \,;
%\quad\text{and}
        \end{align*}
    \item for $1 \leq p < q$ and $2 \leq q$, with
     %     if $1 \leq p < q$ and $2 \leq q$, put
      \begin{align*}
        m &:= \lceil 144 \, K^{p q/(q-p)} \rceil \,,\\
        s &:= 2 \,,\\
        \eta &:= 16 \, K^{p q/(q-p)} \, \eps^{-2} \,,
      \end{align*}
  \end{itemize}
  is $(\eps,\delta)$-approximating on $\Yinput_{p,q,K}$.
\end{thm}

\subsection{Estimates on the cost}

% Obviously, the cost~$n(\omega,Y)$ of the considered algorithms is random.
Obviously, the cost $n(\checked{\omega},Y)$ of the considered algorithms is
a random variable. % direktere Formulierung.
The quantity of interest is the expected cost
\begin{equation*}
 \bar{n}(A_{k,p,m}^{k',s,\eta},Y)
   := \expect n(\,\cdot\,,Y)
   \leq km+k' (1+\eta\, \expect  \widehat{R}_{k,m}^s).
\end{equation*}
The following lemma is essential,
since it provides an upper bound on $\expect \widehat{R}_{k,m}^s$.

\begin{lem} \label{lem:E(R^s)}
  Let $Y \in \Yinput_{p,q,K}$ with $1 \leq p \leq 2$ and $p < q \leq 2p$.
  Let $\rho = \|Y - \expect Y\|_p$
  and $s = \frac{1}{p} \max\{1 + 1/(q-1), 2\}$.
  For $0<\gamma<1$ and $\alpha_1 = 1/4$
  choose~$m$ such that (via Lemma~\ref{lem:|R_m-rho|<})
  for the $p$-moment estimators~$\widehat{R}_m^{(\ell)}$
  from Algorithm~\ref{alg:A_kpmksh} we have
  \begin{equation} \label{eq:|R_m-rho|<gam rho}
   \P\left\{ | \widehat{R}_m^{(\ell)}-\rho^p | \leq \gamma \rho^p \right\}
     \geq 1 - \alpha_1 \,.
  \end{equation}
  Assume that \mbox{$k \geq 4ps/q$}, with $k$ being odd.
  Then the median of $p$-moment estimators, $\widehat{R}_{k,m}$, satisfies
  \begin{equation*}
    \expect \widehat{R}_{k,m}^s
      \,\leq\, (1 + 3 \,\gamma)^s \, \rho^{ps} \,.
  \end{equation*}
\end{lem}

\begin{proof}
  Lemma~\ref{lem:|R_m-rho|<} and Proposition~\ref{prop:|R_km-rho|<}
  are stated for $0<\alpha<1/2$,
  so our assumption~\eqref{eq:|R_m-rho|<gam rho} implies
 \begin{equation*}
  \P\{|\widehat{R}_{k,m} - \rho^p| \leq c \, \alpha^{-p/q} \, \rho^p\}
  \geq 1 - \fracts{1}{2} \, (4 \alpha (1-\alpha))^{k/2} \,,
  \end{equation*}
  where \mbox{$c = \alpha_1^{p/q} \, \gamma = 4^{-p/q} \, \gamma$}.
  We use the simplified estimate
  \begin{equation*}
    \P\left\{|\widehat{R}_{k,m} - \rho^p|^s
                > c^s \, \alpha^{- ps/q} \, \rho^{ps}
      \right\}
        \leq 2^{k-1} \, \alpha^{k/2} \,.
  \end{equation*}
  Suppose $\expect \widehat{R}_{k,m}^s$ is finite,
  then by the triangle inequality applied  to the $\|\cdot\|_s$-norm,
  we obtain
  \begin{equation*}
    \expect \widehat{R}_{k,m}^s
      \leq \left(\rho^{p} + \| \widehat{R}_{k,m} - \rho^p\|_s\right)^s \,.
  \end{equation*}
  The norm therein can be estimated by
  \begin{equation*}
    \expect|\widehat{R}_{k,m} - \rho^p|^s
      \leq t_0 + \int_{t_0}^{\infty}
    \P\{|\widehat{R}_{k,m} - \rho^p|^s > t\}
                  \dint t \,.
  \end{equation*}
  This leads to
   $\alpha
      = c^{q/p} \, \rho^q \, t^{-q/(ps)}$\,.
  We choose
    $t_0 = \alpha_1^{-ps/q} \, c^s \, \rho^{ps}
         = 4^{ps/q} \, c^s \, \rho^{ps}$,
  and we use
  \begin{equation*}
    \int_{t_0}^{\infty} t^{-kq/(2ps)} \dint t
      = \left(\frac{kq}{2ps} - 1\right)^{-1} \, t_0^{-kq/(2ps) + 1} \,,
  \end{equation*}
  which holds for~$kq/(2ps) > 1$, that is~$k > 2ps / q$, only.
  Then
  \begin{align*}
    \int_{t_0}^{\infty}
          \P\{|\widehat{R}_{k,m} - &\rho^p|^s > t\}
    \dint t\\
      &\,\leq\, 2^{k-1} \, c^{kq/(2p)} \, \rho^{kq/2}
          \, \left(\frac{kq}{2ps} - 1\right)^{-1}
          \, (4^{ps/q} \, c^s \, \rho^{ps})^{-kq/(2ps) + 1} \\
      &\,=\, \left(\frac{kq}{2ps} - 1\right)^{-1}
                  \, 2^{2ps/q - 1} \, c^s \, \rho^{ps}\,.
  \end{align*}
  Altogether we obtain
  \begin{align*}
    \expect \widehat{R}_{k,m}^s
      &\,\leq\, \left(1 + \left(1
      + \left(\frac{kq}{2ps} - 1 \right)^{-1} \, \frac{1}{2}
      \right)^{1/s} \, 2^{2p/q} \, c
      \right)^s
                \, \rho^{ps} \, . 
%E                              ^ 
        \end{align*}                  
  Taking~$k \geq 4ps/q$, and with~\mbox{$c \leq 2^{-p/q}\,\gamma$} as well as $p/q<1$,
  this further simplifies to
%   \begin{align*}
$
  \expect \widehat{R}_{k,m}^s
      \,\leq\, \left(1 + 2 \left(3/2\right)^{1/s} \, \gamma
              \right)^s
        \, \rho^{ps} \,.
$
        %   \end{align*}
  We have $s\geq 1$ such that~\mbox{$1/s \leq 1$}, 
  which leads to the assertion.
\end{proof}

With the additional assumption on the median trick that $k\geq 4ps/q$,
we can estimate the expected cost for fixed input.
For the particular setting from Theorem~\ref{thm:solvability}, where we use
$\Yinput_{p,q,K}\subset \Yinput_{1,\wt{q},K'}$
with~$\wt{q}:=\min\{q,2\}$,
we get the following upper bound on the cost.

\begin{thm} \label{thm:costbound}
  For computing %the computation of
  $\expect Y$ on $\Yinput_{p,q,K}$ with
  \mbox{$1 \leq p < q \leq \infty$}, see~\eqref{eq:cone}, consider 
  $A_{k,1,m}^{k',s,\eta}$, see Algorithm~\ref{alg:A_kpmksh}.
        For any \mbox{$\eps > 0$} and \mbox{$0 < \delta < 1/2$} 
  choose $m$, $s$, and $\eta$ as in Theorem~\ref{thm:solvability}, 
  where~$k$ and $k'$ are the least odd natural numbers
  such that \mbox{$k,k' \geq 2 \log \delta^{-1} / \log \fracts{4}{3}$},
  and in addition $k \geq \max\{4/(q-1),4\}$.
  Then, $A_{p,q,K}^{\eps,\delta} := A_{k,1,m}^{k',s,\eta}$
        is $(\eps,\delta)$-approximating on~$\Yinput_{p,q,K}$
  and
  \begin{equation*}
    \bar{n}(A_{p,q,K}^{\eps,\delta},Y)
      \leq C_q \, K^{pq/(q-p)}
            \, \left(1 + \left(\frac{\|Y - \expect Y\|_1}{\eps}
  \right)^{\max\{1+1/(q-1),2\}}
  \right)
  \, \log \delta^{-1}
  \end{equation*}
  where $Y \in \Yinput_{p,q,K}$,
  %with first moment~\mbox{$\rho := \|Y - \expect Y\|_1$},
  and $C_q > 0$ is a $q$-dependent constant.
\end{thm}

Note that in the above result,
$K > 1$ and $\delta < 1/2$ ensure that
additive constants can be absorbed by~$C_q$ and
are not lost by writing $K^{pq/(q-p)}$ and $\log \delta^{-1}$ as factors.

\begin{rem}[Balancing the cost between the two stages] \label{rem:balance}
  There is some arbitrariness about
  how much effort we put on each of the two stages
  for $Y\in\Yinput_{p,q,K}$.
  The more accurately we estimate $\|Y - \expect Y\|_1$,
  that is with relative accuracy $\gamma$ tending to $0$,
  the more accurately we can choose the sample size~$n_2$ for the second stage.
  In return, being parsimonious in the first stage,
  which means $\gamma \to 1$,
  we likely overshoot the sample size for the second stage.
  We thus suggest the following compromise:
  Choose~$\gamma \in (0,1)$ such that the cost of each of the two stages
  exceeds the cost of the respective limiting cases
  by roughly the same factor.
  
  The algorithm for which we give cost bounds in Theorem~\ref{thm:costbound}
  uses an estimate for the first moment.
  With this strategy, for any input set~$\Yinput_{p,q,K}$,
  the cost for the first stage possesses the optimal dependence on~$K$,
  compare the lower bounds in Theorem~\ref{thm:fixedcost}.
  But since actually a higher moment
  is of interest for the sample size selection in the second stage,
  namely $\expect |Y-\expect Y|^{\wt{q}}$,
  with $\wt{q} := \min\{q,2\}$,
  compare Proposition~\ref{prop:M_km-prob},
  the inflation parameter~$\eta$ has to be proportional to~$K^{pq/(q-p)}$.
  For some instances \mbox{$Y \in \Yinput_{p,q,K}$}
  with a far smaller ratio between the different moments
  we thus actually overshoot the sample size by this factor.
  It might be hence beneficial for the adaption
  to estimate higher moments within the first stage:
  \begin{itemize}
    \item For %input sets %% Robert: Aus Platzgr\"unden weggelassen
      $\Yinput_{p,q,K}$ with $1 \leq p < q \leq 2$,
      estimate the $p$-moment \mbox{$\expect|Y-\expect Y|^p$}.
      The cost for the first stage has still the optimal dependence on~$K$,
      the inflation factor~$\eta$ for the second stage
      is only proportional to~$K^{q/(q-1)}$.
    \item For input sets $\Yinput_{p,q,K}$ with $2 < q$, estimate the variance.
      That way we directly estimate the moment that is decisive
      for the sample size in stage two.
      The cost for stage one, though, will have a suboptimal dependence on~$K$
      if~$p < 2$.
  \end{itemize}
\end{rem}

\begin{rem} \label{rem:VarEst}
  Hickernell et al.\ consider in \cite{HJLO13}
  the expectation problem on $\Yinput_{2,4,K}$.
  They provide an $(\eps,\delta)$-approximating
  two-stage algorithm on this special class of inputs
  and that way prove solvability.
  The main difference to our Algorithm~\ref{alg:A_kpmksh}
  is that they do not use the median trick,
  which is tantamount to the parameter choice~$k=k'=1$.
  The cost of their algorithm then depends polynomially on
  the uncertainty~$\delta$
  instead of the optimal logarithmic dependence~\mbox{$\log \delta^{-1}$}
  which we have.
  
  For the special setting~$\Yinput_{2,4,K}$,
  the constants of our complexity bounds
  can be slightly improved.
  In detail, there is an unbiased empirical variance estimator
  \begin{equation*}
    \widehat{V}_m := \frac{1}{m-1} \sum_{i=1}^m (Y_i - \widehat{M}_m)^2,\qquad m\geq 2,
  \end{equation*}
  which shall replace~$\widehat{R}_m$, see~\eqref{eq:R_m}, % with~$p=2$,
  and its variance \mbox{$\Var(\widehat{V}_m)$}
  can be precisely estimated since~$Y$ possesses a fourth moment.
  Then, Lemma~1 in Hickernell et al.~\cite{HJLO13}
  can be used instead of Lemma~\ref{lem:|R_m-rho|<} in our paper
  for the confidence level~$\alpha_1$ before probability amplification.
  In combination with the median trick this enables us
% Daniel: Nur bessere Konstanten?
% Robert: Dass die Konstanten besser werden,
%         steht bereits am Anfang des Absatzes.
%         Der Absatz davor hat schon erklärt
%         dass der Median-Trick eine bessere $\delta$-Abhängigkeit ermöglicht.
%         Ich sehe also den Satz als Zusammenfassung des bereits gesagten.
  to construct algorithms with more favorable cost bounds.
% \dr{in the constants}. %Robert: Braucht es meiner Ansicht nach nicht.
\end{rem}

\section{Lower bounds}   \label{sec:LBs}

\subsection{Algorithms, solvability, and complexity}

We start with a formal definition of algorithms
that are admissible for 
the computation of $\expect Y$
(or other quantities associated to a random variable~$Y$,
e.g.\ different moments),
as well as the larger class of algorithms
admissible for numerical integration %~\eqref{eq:INT}
of a function~$f$.
%%Robert: Bei $\expect Y$ hatten wir auf nicht auf \eqref{eq:E(Y)} verwiesen.
Recall that $(\Omega,\Sigma,\P)$ is a probability space.
In the context of algorithms that take random variables as input,
let $\Yinput$ be a subset of all real-valued random variables
on $(\Omega,\Sigma,\P)$.
For any input~$Y \in \Yinput$
we assume that we have a sequence of i.i.d.\ copies $Y_1,Y_2,\ldots \sim Y$
defined on $(\Omega,\Sigma,\P)$.
The algorithms may depend on an additional random variable~$U$
with values in a measurable space $(\mathbb{U},\Sigma_U)$
which influences how the i.i.d.\ sample is used.
This random variable is assumed to be independent
of the i.i.d.\ sample $(Y_i)_{i\in\N}$.
In contrast, for algorithms that take functions as input,
the randomness is purely created by the algorithm design,
whereas the input class~$\Finput = \Finput(\pi)$
is a subset of all measurable functions~$f\colon G\to \R$
on the probability space $(G,\Sigma_G,\pi)$
and algorithms may interrogate %use
function values~$f(\vecx)$ at any point~$\vecx \in G$.

\begin{defi}[Abstract algorithms] \label{def:alg}
  Assume that we have
  \begin{itemize}
    \item an auxiliary random variable $U$ with values in a measurable space
      $(\mathbb{U},\Sigma_U)$;
    \item a measurable output function
      $\phi\colon \mathbb{U} \times \R^{\N} \to \R$; and
    \item a measurable function
      $n\colon \mathbb{U} \times \R^{\N} \to \N_0 \cup \{\infty \}$
      (called \emph{stopping rule})
      for which the value of \mbox{$\ind[n(u,(y_j)_{j \in \N}) \geq i]$}
      only depends on~$(u,y_1,\ldots,y_{i-1})$.
      % Robert: Ich hab mir mal erlaubt, das anzupassen,
      %         damit es besser zu unten passt, vorher hieß das Ereignis "n > i".
  \end{itemize}
  Based on this we define two types of algorithms:
  \begin{enumerate}
    \item An \emph{i.i.d.-based algorithm}~$A$ on~$\Yinput$
      is a mapping \mbox{$A: \Omega \times \Yinput \to \R$}
      given by
      \begin{equation*}
        A(\omega,Y)
          := \phi(U(\omega),
                  Y_1(\omega),\ldots,Y_{n(\omega,Y)}(\omega),0,0,\ldots) \,,
      \end{equation*}
      where $(Y_i)_{i \in \N}$ is a sequence
      of i.i.d.\ copies of~$Y \in \Yinput$
      and independent of~$U$;
      further, the \emph{stopping time}
      \begin{equation*}
        n(\omega,Y) := n(U(\omega),(Y_i(\omega))_{i \in \N})
      \end{equation*}
      is assumed to be $\P$-almost surely finite for all $Y \in \Yinput$.
    \item A \emph{general randomized algorithm}~$Q$ on~$\Finput$
      is a mapping \mbox{$Q: \Omega \times \Finput \to \R$}
      given by
      \begin{equation*}
        Q(\omega,f) := \phi(U(\omega),N(\omega,f))
      \end{equation*}
      with a measurable information map $N: \Omega \times \Finput \to \R^{\N}$,
      \mbox{$N(\omega,f) = (y_i)_{i \in \N}$}, satisfying
      \begin{equation*}
        y_i = \begin{cases}
                f(\vecX_i(\omega,y_1,\ldots,y_{i-1}))
                  &\quad\text{for $n(U(\omega),(y_j)_{j \in \N}) \geq i$,}\\
                0 &\quad \text{for $n(U(\omega),(y_j)_{j \in \N}) < i$,}
              \end{cases}
      \end{equation*}
      for a sequence~$(\vecX_i)_{i\in \N}$ of measurable functions
      $\vecX_i: \Omega \times \R^{i-1} \to G$
      (called \emph{sampling nodes});\footnote{
        The $y_i$ are well defined since by the definition of 
        a stopping rule
        the outcome of the statement $n(U(\omega),(y_j)_{j \in \N}) \geq i$ depends only on $U(\omega)$
        and $y_1,\ldots,y_{i-1}$.
      }
      % Daniel: Ich finde das man bei den $y_i$s stolpern kann, da ja die Fallunterscheidung scheinbar von der ganzen Folge der $y_j$s abhaengt. Deshalb die Fussnote als Vorschlag.
% Robert: An sich gut. Mal abwarten, was das Journal dazu sagt. Erich meinte mal, die mögen solche Fußnoten nicht. Aber dadurch, dass das einen Einzelfall darstellt, könnte man das dann leicht wieder ändern, wenn es gewünscht wird.
      further, the \emph{stopping time}
      (also called \emph{cardinality function})
      \begin{equation*}
        n(\omega,f) := n(U(\omega),N(\omega,f))
      \end{equation*}
      is assumed to be $\P$-almost surely finite for all $f \in \Finput$.
  \end{enumerate}
  Thus, the application of an algorithm to an input,
  that is \mbox{$A(Y) := A(\,\cdot\,,Y)$} or \mbox{$Q(f) := Q(\,\cdot\,,f)$},
  respectively,
  is a random variable on $(\Omega,\Sigma,\P)$.
\end{defi}
Every i.i.d.-based algorithm provides a general randomized algorithm:
For the integration problem \eqref{eq:INT}
on a class~$\Finput = \Finput(\pi)$ of $\pi$-measurable
real-valued functions,
let $\vecX$ be a $\pi$-distributed random variable and define the input class
\mbox{$\Yinput := \{f(\vecX)\mid f\in \Finput\}$}.
For every $f\in\Finput$ and an 
i.i.d.\ sequence of $\pi$-distributed sampling nodes $(\vecX_i)_{i\in\N}$
\checked{independent of~$U$}, 
observe that $(f(\vecX_i))_{i\in\N}$ 
is an i.i.d.\ sequence of copies of \mbox{$f(\vecX) \in \Yinput$},
thus applicable within any i.i.d.-based algorithm~$A$.
That way, $A(\omega,f(\vecX))$ can be interpreted as a general randomized algorithm
$Q(\omega,f)$ with restricted type of information.
Lower bound proofs are less technical
  for the smaller class of i.i.d.-based algorithms,
  see Sections~\ref{sec:Y-LBs0} and \ref{sec:Y-worLBs},
  but can be extended to general integration methods,
  see Section~\ref{sec:INTLBs}.
\checked{Let us mention here that we deviate from the ``standard'' technique
% of Bakhvalov \cite{Bakh59} 
of proving lower bounds for general randomized algorithms in IBC
via an average case setting, 
see for example \cite[Lemma~4.37]{NW08}.
Such a direct approach for proving lower bounds for the integration
\checked{with ``small error at high confidence''
is developed in~\cite{KR18} with application to Sobolev classes.}}
%on Sobolev classes
%for ``small error with high probability'' is developed 
%in \cite{KR18}.}

Before proceeding with our results,
we formally introduce the concepts of solvability and complexity
for the expectation problem \eqref{eq:E(Y)}
on a class $\Yinput$ and,
analogously, for the integration problem \eqref{eq:INT}
on a class $\Finput$.
  
\begin{defi}[Solvability]
  The problem~\eqref{eq:E(Y)} of computing an expected value
  is \emph{solvable} on a class of random variables~$\Yinput$,
  iff for any $\eps > 0$, and $\delta \in (0,1)$,
  there is an i.i.d.-based algorithm~$A$ such that
  \begin{equation*}
    \P\{ | A(Y) - \expect Y | \leq \eps \} \geq 1-\delta,
  \end{equation*}
  for all $Y\in\Yinput$.
  In that case $A$ is called \emph{$(\eps,\delta)$-approximating
  \checked{(on $\Yinput$)}}.
  
  The integration problem \eqref{eq:INT}
  is \emph{solvable} on $\Finput$ 
  iff for any $\eps > 0$, and $\delta \in (0,1)$,
  there is a general randomized algorithm~$Q$ such that
  \begin{equation*}
    \P\{ | Q(f) - \Int(f) | \leq \eps \} \geq 1-\delta,
  \end{equation*}
  for all $f\in\Finput$.
  Here $Q$ is also called \emph{$(\eps,\delta)$-approximating
  \checked{(on $\Finput$)}}.
\end{defi}
\begin{rem}[Other notions of solvability]
  If a problem is even solvable in the sense
  that for any error threshold $\eps > 0$
  we can find an algorithm~$A$ that guarantees
  a bound for the \emph{expected error}
  \mbox{$\expect |A(Y) - \expect Y| \leq \eps$}
  (or the \emph{mean squared error}
  \mbox{$\expect |A(Y) - \expect Y|^2 \leq \eps^2$}),
  using Markov's inequality we can also guarantee solvability
  in the probabilistic sense of the definition above.
  The reverse statement is \emph{not} true.
  In fact, Theorem~\ref{thm:E(err)<eps--noway} shows
  that the notions differ for certain problems.
\end{rem}
Besides solvability of the problem,
for an $(\eps,\delta)$-approximating algorithm $A$ (resp.\ $Q$)
with $Y\in\Yinput$ (resp.\ $f\in\Finput$) 
we are interested in the cost,
given by the sample size $n(\checked{\omega},Y)$ (resp.\ $n(\checked{\omega},f)$).
Note that $n(\,\cdot\,,Y)$ is a random variable that depends on~$Y$, for example,
it might depend on the $p$-moment of~$Y$.
For the complexity analysis
we aim to control the \emph{expected cardinality},
\begin{equation} \label{eq:En}
  \bar{n}(A, Y)
    := \expect n(\,\cdot\,,Y)
    \qquad (\text{resp.} \; \bar{n}(Q, f)
    := \expect n(\,\cdot\,,f)) \,,
\end{equation}
so that running the algorithm multiple times
for several random variables of comparable difficulty
does not take an uncontrollably long computing time.

For classical input sets~$\Yinput$ (resp.\ $\Finput$) such as $L_q$-balls,
one usually considers the worst case cost
\begin{equation}
  \bar{n}^{\wor}(A, \Yinput) := \sup_{Y \in \Yinput} \bar{n}(A,Y) 
  \qquad (\text{resp.}\;
  \bar{n}^{\wor}(Q, \Finput) := \sup_{f \in \Finput} \bar{n}(Q,f) 
)
  \,.
\end{equation}
For $\eps,\delta>0$ the \emph{minimal worst case cost}
(or simply the \emph{worst case complexity} of the problem)
is given by
\begin{align*}
  \bar{n}^{\wor}(\eps,\delta,\Yinput)
    & := \inf \{n > 0 \mid \exists \;(\eps,\delta){\text{-approximating}\;} A\;
%       e_{\delta}^{\prob}(A,\Yinput) \leq \eps, \;
        \text{with}\;  \bar{n}^{\wor}(A,\Yinput) \leq n
            \},\\
  (\text{resp.}\;\bar{n}^{\wor}(\eps,\delta,\Finput)
  & := \inf \{n > 0 \mid \exists \;(\eps,\delta){\text{-approximating}\;} Q\;
%       e_{\delta}^{\prob}(A,\Yinput) \leq \eps, \;
        \text{with}\;  \bar{n}^{\wor}(Q,\Finput) \leq n
            \}).
\end{align*}
However,
for the cone-shaped input sets $\Yinput = \Yinput_{p,q,K}$ we are interested in,
this quantity is infinite which can be seen
via an appropriate lower bound proof,
compare Theorems~\ref{thm:LB-YcapB2} and~\ref{thm:LB-YcapBq}
and consider the limit of intersections of~$\Yinput_{p,q,K}$
with bigger and bigger $L_q$-balls.
Hence, we need an alternative concept of complexity.
We take two perspectives here.

First, we consider the \emph{fixed cost} of an i.i.d.-based algorithm~$A$ 
(resp.\ a general randomized algorithm $Q$), that is,
\begin{equation}
  \bar{n}^{\fix}(A, \Yinput) := \inf_{Y \in \Yinput} \bar{n}(A,Y), 
  \qquad (\text{resp.}\;    \bar{n}^{\fix}(Q, \Finput) := \inf_{f \in \Finput} \bar{n}(Q,f))\,.
\end{equation}
For $\Yinput = \Yinput_{p,q,K}$ and Algorithm~\ref{alg:A_kpmksh},
this is essentially the cost for estimating the \mbox{$p$-moment}.
Given any problem, we define its
\emph{minimal fixed cost}
by
\begin{align*}
  \bar{n}^{\fix}(\eps,\delta,\Yinput)
    &:= \inf\{n > 0 \mid \exists \;(\eps,\delta){\text{-approximating}\;} A\;
                          \text{with}\;  \bar{n}^{\fix}(A,\Yinput) \leq n
            \},\\
  (\text{resp.}\;
    \bar{n}^{\fix}(\eps,\delta,\Finput)
      & := \inf \{n > 0 \mid \exists \;(\eps,\delta){\text{-approximating}\;} Q\;
                            \text{with}\;  \bar{n}^{\fix}(Q,\Finput) \leq n
                \}),
\end{align*}
which is the fixed cost
that \emph{all} $(\eps,\delta)$-approximating %i.i.d.-based
%% Robert:
%% i.i.d.-based nicht überall mitführen, wenn Definitionen für beides gelten
algorithms exhibit,
not only those that start with a $p$-moment estimation,
see Theorem~\ref{thm:fixedcost} for~\mbox{$\Yinput = \Yinput_{p,q,K}$}.

Second, we intersect the cone shaped input sets~$\Yinput_{p,q,K}$
with (semi-)balls of radius~\mbox{$\tau > 0$}
with respect to the centered $L_{\wt{q}}$-norm,
$\wt{q} := \min\{q,2\}$, that is,
\begin{align*}
  \tau\Barrel_{\wt{q}}
    & := \{Y \in L_{\wt{q}} \mid \|Y - \expect Y\|_{\wt{q}} \leq \tau\} \,,\\
 (\text{resp.}\; \tau\Barrel_{\wt{q}}
    & := \{f \in L_{\wt{q}} \mid \|f - \Int(f)\|_{\wt{q}} \leq \tau\} \,,)
\end{align*}
and consider the worst case complexity,
\begin{equation*}
  \bar{n}^{\wor}(\eps,\delta,\Yinput_{p,q,K} \cap \tau\Barrel_{\wt{q}}),
  \qquad (\text{resp.}\;
  \bar{n}^{\wor}(\eps,\delta,\Finput_{p,q,K} \cap \tau\Barrel_{\wt{q}}),)
\end{equation*}
which is the worst case cost of optimal algorithms
that are adapted to this reduced problem.
With
  \mbox{$\|Y - \expect Y\|_1 \leq \|Y - \expect Y\|_{\wt{q}} \leq \tau$},
  our particular algorithm~$A_{\eps,\delta}^{p,q,K}$,
  see Theorem~\ref{thm:costbound},
  provides an upper bound for this quantity.
  Lower bounds are given
  in Theorems~\ref{thm:LB-YcapB2} and~\ref{thm:LB-YcapBq}.

\subsection{Unsolvability results and fixed cost bounds}
\label{sec:Y-LBs0}

We start with the statement
that higher integrability alone is not a sufficient assumption,
the corresponding problem is unsolvable.

\begin{thm}[Unsolvability for the whole space~$L_q$]
\label{thm:Lq--noway}
  The problem~\eqref{eq:E(Y)} of computing the expectation via i.i.d.-based
  algorithms is unsolvable for the class~$\Yinput_q = L_q$
  of $q$-integrable random variables, where~$q \in [1,\infty]$.
\end{thm}
\begin{proof}
  Let~$A$ be an arbitrary i.i.d.-based algorithm for~$\Yinput_q$,
  further let~$\eps > 0$ and $\delta \in (0,1)$.
  By definition, the i.i.d.-based algorithm $A$ terminates almost
  surely for any input, in particular, 
  for the zero random variable~\mbox{$\mathit{0} \in \Yinput_q$}.
  Hence, for the stopping time $n(\checked{\omega},\mathit{0})$ of $A$
  with input $\mathit{0}$
  we have
  \begin{equation*}
    \P\{n(\,\cdot\,,\mathit{0}) > n_0\}
      \xrightarrow[n_0 \to \infty]{} 0 \,.
  \end{equation*}
%  Choose~$n_0$ such that \mbox{$\P\{n(\,\cdot\,,\mathit{0}) > n_0\} < \delta/2$}.
  
  Consider Bernoulli random variables~$Z$ with probabilities
  \mbox{$\P\{Z = 0\} = 1-p$} and \mbox{$\P\{Z = 1/p\} = p \in (0,1)$},
  hence~$\expect Z = 1$.
  For any~$a > 0$ we have \mbox{$a Z \in \Yinput_q$}
  since it is a discrete random variable.
  We choose~$a > 2\eps$
  such that~\mbox{$|A(Y) - a| \leq \eps$} and \mbox{$|A(Y)| \leq \eps$}
  are disjoint events.
  For i.i.d.\ copies~\mbox{$Z_1,Z_2,\ldots \sim Z$} we have
  \begin{equation*}
    \P\{Z_1 = \ldots = Z_{n_0} = 0\} = (1-p)^{n_0}
      \xrightarrow[p \to 0]{} 1 \,.
  \end{equation*}
%  hence we can choose~$p$
%  such that \mbox{$\P\{Z_1,\ldots,Z_{n_0} = 0\} \geq 1 - \delta/4$}.
  Now, consider the failing probability for a random variable~$aZ$,
  \begin{align*}
    \P\{|A(aZ) - &a| > \eps\}\\
      &\;\geq\;
        \P\{n(\,\cdot\,,aZ) \leq n_0 \text{ and } Z_1=\ldots=Z_{n_0} = 0\}\\
      &\qquad
        \cdot \P(|A(aZ) - a| > \eps
                 \mid n(\,\cdot\,,aZ) \leq n_0 \text{ and } Z_1=\ldots=Z_{n_0} = 0)
                   \\
      &\;\geq\;
        \P\{n(\,\cdot\,,\mathit{0}) \leq n_0\}
          \cdot \P\{Z_1=\ldots=Z_{n_0} = 0\} \\
      &\qquad
          \cdot \P(|A(\mathit{0})| \leq \eps
                  \mid n(\,\cdot\,,\mathit{0}) \leq n_0) \\
      &\;=\;
        \P\{|A(\mathit{0})| \leq \eps \text{ and }
            n(\,\cdot\,,\mathit{0}) \leq n_0\}
          \cdot \P\{Z_1=\ldots=Z_{n_0} = 0\} \\
      &\;\geq\;
        (\P\{|A(\mathit{0})| \leq \eps\} - \P\{n(\,\cdot\,,\mathit{0}) > n_0\})
          \cdot (1-p)^{n_0} \\
      &\xrightarrow[p\to 0]{}
        \P\{|A(\mathit{0})| \leq \eps\} - \P\{n(\,\cdot\,,\mathit{0}) > n_0\} \\
      &\xrightarrow[n_0 \to \infty]{} \P\{|A(\mathit{0})| \leq \eps\} \,.
%      &\;>\;
%        \left(1-\frac{\delta}{2}\right)
%          \left(1-\frac{\delta}{4}\right)
%            \cdot \P(|A(\mathit{0})| \leq \eps
%                     \mid n(\,\cdot\,,\mathit{0}) \leq n_0) \,.
  \end{align*}
  Hence, for any admitted uncertainty~$\delta \in (0,1/2)$,
  if~$\P\{|A(\mathit{0})| \leq \eps\} \geq 1-\delta > 1/2$,
  there are random variables of type~$aZ$ for which
  \mbox{$\P\{|A(aZ) - a| > \eps\} > 1/2 > \delta$}.
  This shows that the algorithm cannot be $(\eps,\delta)$-approximating.
\end{proof}

We now turn to more restricted input classes.
First, we stress that for the cone-shaped input classes~$\Yinput_{p,q,K}$
the worst case
\emph{expected error} is unbounded for any algorithm.
(A fortiori, the same does hold for the mean squared error.)
\begin{thm}[Impossibility of \checked{finite} worst case expected error]
  %$L_r$ error] 
  \label{thm:E(err)<eps--noway}
  For any i.i.d.-based algorithm $A$ on $\Yinput_{p,q,K}$,
  with $1 \leq p < q \leq \infty$ and~$K \geq 1$, we have
  \begin{equation*}
   \sup_{Y \in \Yinput_{p,q,K}} \expect|A(Y) - \expect Y|
                                    %\|\expect Y - A(Y)\|_r
        = \infty \,.
  \end{equation*}
\end{thm}
\begin{proof}
  Consider a random variable~$Z$ with
  \mbox{$\P\{Z = 0\} = \P\{Z = 2\} = 1/2$} and
  observe that~$\expect Z = 1$ as well as $\|Z-1\|_p = 1$
  for all~\mbox{$1 \leq p \leq \infty$}.
  Consequently, for any $a\in \R$ 
  we have~\mbox{$a Z \in \Yinput_{p,q,K}$}
  for all~\mbox{$1 \leq p<q \leq \infty$} and \mbox{$K \geq 1$}.
  We reduce the problem to distinguishing $aZ$
  from the zero random variable \mbox{$\mathit{0} \in \Yinput_{p,q,K}$},
  \begin{align*}
    \sup_{Y \in \Yinput_{p,q,K}} \expect|A(Y)- \expect Y|
                                     %\|\expect Y - A(Y)\|_r 
     & \geq \max\{\expect |A(\mathit{0})|,\expect|A(aZ) - a|\} \\
     & \geq \frac{1}{2} \, \expect| A(\mathit{0}) - A(aZ) + a | \,.
  \end{align*}
  By definition, the i.i.d.-based algorithm $A$ terminates almost
  surely for any input, in particular, 
  for the zero random variable~\mbox{$\mathit{0} \in \Yinput_{p,q,K}$}.
  Hence, for the stopping time
  $n(\checked{\omega},\mathit{0})$ of $A$ with input $\mathit{0}$, 
  there exists an $n_0\in \N$ such that
  \mbox{$\P\{n(\,\cdot\,,\mathit{0}) \leq n_0\} > 0$}.
  Furthermore,
  \begin{align*}
   \expect| A(\mathit{0}) - A(aZ) + a |
%     &\;\geq\;
%       |a| \cdot \P\{A(\omega,\mathit{0}) = A(\omega,aZ)\} \\
     &\;\geq\;
       |a| \cdot \P\{A(\mathit{0}) = A(aZ)\} \\
     &\;\geq\;
       | a| \cdot
               \P\left\{ n(\,\cdot\,,\mathit{0})\leq n_0 \;\text{ and }\; 
                         Z_1=\ldots=Z_{n_0}=0\right\}\\
     &\;=\;
       | a|
         \cdot \P\{n(\,\cdot\,,\mathit{0})\leq n_0\} 
         \cdot \P\{Z_1=\ldots=Z_{n_0}=0\}.
  \end{align*}
  Finally, the assertion follows 
  by the fact that $\P\{Z_1=\ldots=Z_{n_0}=0\}=2^{-n_0}$ and
  \begin{equation*}
    \sup_{Y \in \Yinput_{p,q,K}} \expect|A(Y) - \expect Y | 
      \;\geq\;
       \frac{| a | }{2^{n_0+1}} \, \P\{n(\,\cdot\,,\mathit{0})\leq n_0\}
      \xrightarrow[| a| \to \infty]{} \infty.
  \qedhere
  \end{equation*}
\end{proof}

The previous theorem
tells us that certain standard worst case error notions %, %terms,
such as the expected absolute error or the mean squared error 
are not suitable for cone-shaped input classes~$\Yinput_{p,q,K}$.
This is the reason for us to focus on the concept of
``small error with high probability''.
We start with an auxiliary result.

\begin{lem} \label{lem:Bernoulli,K>>1}
For a Bernoulli random variable~$Y$
  with \mbox{$\P\{Y = 0\} = 1 - a$}
  and \mbox{$\P\{Y = 1\} = a \in (0,1/2]$}
  satisfying
  \begin{equation} \label{eq: gen_a} 
    a \geq K^{-pq/(q-p)}
    \qquad \text{or} \qquad
    a = \fracts{1}{2} \,,
  \end{equation}
  we have \mbox{$Y \in \Yinput_{p,q,K}$}
  for~$1 \leq p < q \leq \infty$ and~$K \geq 1$.
\end{lem}
\begin{proof}
  Clearly, $\expect Y = a$, such that
  we need to study 
  \begin{equation*}
    \kappa_{p,q}(a) := \frac{\|Y-a\|_q}{\|Y-a\|_p} \,,
    \qquad a \in (0,1/2] \,.
  \end{equation*}
  The case~$a=1/2$ is explained by
  $\kappa_{p,q}(1/2)=1$.
  For $a \in (0,1/2]$ we have
  \begin{equation*}
    \kappa_{p,\infty}(a)
      = \frac{1-a}{\left((1-a) a^p + a(1-a)^p\right)^{1/p}}
      \leq a^{-1/p} \,.
  \end{equation*}
  For~$p < q < \infty$ we interpolate the $L_q$-norm,
  see Lemma~\ref{lem:RVprq},
  \begin{equation*}
    \kappa_{p,q}(a)
      \leq (\kappa_{p,p}(a))^{1-\lambda}
            (\kappa_{p,\infty}(a))^{\lambda} \,,
  \end{equation*}
  where~$\lambda = (q-p)/q$,
  and \mbox{$\kappa_{p,p}(a) = 1$} by definition.
  Hence
    $\kappa_{p,q}(a) \leq a^{-(q-p)/(pq)}$. %-(1/p-1/q) 
  By inversion we obtain the sufficient condition on~$a$
  such that \mbox{$\kappa_{p,q}(a) \leq K$},
  which is the cone condition~$Y \in \Yinput_{p,q,K}$.
\end{proof}

The next theorem shows that -- up to a $q$-dependent constant --
the fixed cost for the first stage of
Algorithm~\ref{alg:A_kpmksh} is really needed,
compare the term in Theorem~\ref{thm:costbound}
that is independent of the error threshold~$\eps$.

\begin{thm}[Fixed cost] \label{thm:fixedcost}
%  For any $(\eps,\delta)$-approximating
%  i.i.d.-based algorithm~$A$ on~$\Yinput_{p,q,K}$,
%  with~\mbox{$1 \leq p < q \leq \infty$} and~$K > 1$,
%  we have the following fixed cost lower bound 
  For any~$\eps > 0$ and $\delta \in (0,1)$,
  we have the following fixed cost lower bound
  for i.i.d.-based algorithms on~$\Yinput_{p,q,K}$
  with~\mbox{$1 \leq p < q \leq \infty$} and~$K > 1$,
  \begin{equation*}
%    \bar{n}^{\fix}(A,\Yinput_{p,q,K})
    \bar{n}^{\fix}(\eps,\delta,\Yinput_{p,q,K})
      \,\geq\, \frac{1}{2 \log 2} \, K^{pq/(q-p)} \, \log \delta^{-1} \,.
  \end{equation*}
\end{thm}
\begin{proof}
  Let $Y \in \Yinput_{p,q,K}$ and $A$ be an $(\eps,\delta)$-approximating 
  i.i.d.-based sampling algorithm with stopping time $n(\checked{\omega},Y)$. 
  To shorten the notation let
  \mbox{$\bar{n} := \bar{n}(A,Y)$} be the expected cardinality for~$Y$.
  
  Consider a Bernoulli random variable~$Z$
%  independent of~$Y$,
%% Robert: Unabhängigkeit wird nicht gebraucht.
  with probabilities \mbox{$\P\{Z = 1\} = a$} and
  \mbox{$\P\{Z = 0\} = 1 - a$},
  where \mbox{$a=\expect Z \in (0,1/2]$} is chosen such that
  \begin{equation*}
    \|Z - \expect Z\|_q < K \, \|Z - \expect Z\|_p \,.
  \end{equation*}
  For~$c > 0$ let $Y' := Y + c \, Z$.
  Observe that for sufficiently large~$c$ we have
  \begin{equation*}
    \frac{\|Y' - \expect Y'\|_q}{\|Y' - \expect Y'\|_p}
      \stackrel{\text{$\Delta$-ineq.}}{\leq}
        \frac{c \, \|Z - \expect Z\|_q + \|Y - \expect Y\|_q
            }{c \, \|Z - \expect Z\|_p - \|Y - \expect Y\|_p}
      \xrightarrow[c \to \infty]{}
        \frac{\|Z - \expect Z\|_q}{\|Z - \expect Z\|_p}
      \,<\, K \,,
  \end{equation*}
  hence, there is a number~$c_0 > 0$ such that
  \mbox{$Y' \in \Yinput_{p,q,K}$} for~\mbox{$c > c_0$}.

  There is a certain chance
  that $A$ cannot distinguish $Y$ from~$Y'$,
  namely,
  if~$Z_i=0$ for all generated i.i.d.\ samples \mbox{$Y'_i = Y_i + c \, Z_i$}.
  Via Jensen's inequality and by
  exploiting the fact that the function~\mbox{$x \mapsto (1-a)^x$} is convex,
  we have
  \begin{align*}
    \P\{Z_1 = \ldots = Z_{n(\,\cdot\,,Y)} = 0\}
      &= \sum_{n=0}^{\infty} \P\{n(\,\cdot\,,Y) = n
          \;\text{ and }\;
        Z_1 = \ldots = Z_n = 0\} \\
      &= \sum_{n=0}^{\infty} \P\{n(\,\cdot\,,Y) = n\} \, (1-a)^n 
      \geq (1-a)^{\bar{n}} \,.
  \end{align*}
  The algorithm~$A$ returns only finite values, thus, for any 
  $\theta\in(0,(1-a)^{\bar{n}})$
  we find a value~$b \in \R$ such that
  \begin{equation*}
    \P\{A(Y) \leq b \;\text{ and }\; Z_1 = \ldots = Z_{n(\,\cdot\,,Y)} = 0\}
      \geq (1-a)^{\bar n}-\theta \,.
  \end{equation*}
  We want to compute
  \mbox{$\expect Y' = \expect Y + c \, a$},
  which can be arbitrarily large when~$c \to \infty$. Now,
  choose $c$ such that $\expect  Y'-b>\eps$ 
  for an arbitrary error threshold~$\eps > 0$. Then
  \begin{align*}
   \P\left\{ | A(Y') - \expect Y'| >\eps \right\}
   & \geq 
   \P\{A(Y) \leq b \;\text{ and }\; Z_1 = \ldots = Z_{n(\,\cdot\,,Y)} = 0\}\\
   &\geq (1-a)^{\bar n}-\theta. 
  \end{align*}
        Letting $\theta\to 0$ leads to   
        $\P\left\{ | A(Y') - \expect Y'| >\eps \right\} \geq 
        (1-a)^{\bar n}$ such that
  the uncertainty~$\delta$ is at least~$(1-a)^{\bar{n}}$.
  In other words,
  \begin{equation*}
    \bar{n} \,\geq\, \frac{\log \delta^{-1}}{\log(1-a)^{-1}}
      \,\geq\, \frac{1}{2 \log 2} \, a^{-1} \, \log \delta^{-1} \,.
  \end{equation*}
  Here, we used \mbox{$\log (1-a)^{-1} \leq 2 (\log 2) \, a$}
  for~\mbox{$a \in (0,1/2]$}.
  In particular with~\mbox{$a = 1/2$} this is
  \begin{equation} \label{eq:n0>clogd}
    \bar{n} \,\geq\, \frac{\log \delta^{-1}}{\log 2} \,.
  \end{equation}
  For~$K \geq 2^{(q-p)/(pq)}$ %$K \geq 2^{1/p-1/q}$
  we put~\mbox{$a = K^{-pq/(q-p)}$},
  see Lemma~\ref{lem:Bernoulli,K>>1},
  which gives
  \begin{equation} \label{eq:n0>cKlogd}
    \bar{n} \,\geq\, \frac{1}{2 \log 2} \, K^{pq/(q-p)} \, \log \delta^{-1} \,.
  \end{equation}
  Trivially, this holds true for smaller~$K > 1$ as well,
  since in these cases inequality~\eqref{eq:n0>cKlogd}
  is weaker than~\eqref{eq:n0>clogd}.
\end{proof}

\subsection{Worst case complexity for bounded central norm}
\label{sec:Y-worLBs}

In this section we provide lower bounds of the worst case complexity 
for i.i.d.-based algorithms on cone-shaped input classes intersected with
(semi-)balls. 
We start with an auxiliary result due to Wald,
see \cite[Theorem on p.~156]{Wa45}, 
which provides a lower bound for the complexity of statistical tests.

\begin{prop}[Wald~1945] \label{prop:Wald-probrat}
  Let $Y^{(1)}$ and $Y^{(2)}$ be two discrete random variables
  that only take values in a finite set $D \subset \R$,
  and consider the probability ratio
  as a function
  \begin{equation*}
    r(y) := \frac{\P\{Y^{(1)} = y\}}{\P\{Y^{(2)} = y\}}
    \qquad \text{for $y \in D$.}
  \end{equation*}
  Let~$T$ be an i.i.d.-based algorithm
%% Robert: Folgendes ergibt sich automatisch, gibt ja auch noch ne Bemerkung
%  that returns either~$1$ or $2$,
  satisfying for \mbox{$0 < \delta < 1/2$} that
  \begin{equation*}
    \P\{T(Y^{(i)}) \not= i\} \leq \delta \,, \qquad \text{for $i=1,2$.}
  \end{equation*}
  ($T$ is a statistical test
    for hypotheses $1$ vs.\ $2$ with confidence $1-\delta$.)
  Then
  \begin{equation*}
    \bar{n}(T,Y^{(i)})
      \geq \frac{(1-2\delta) \log \frac{1-\delta}{\delta}
                }{|\expect \log r(Y^{(i)})|}
                , \qquad \text{for $i=1,2$.}
  \end{equation*}
\end{prop}

We start with lower bounds for cones of random variables with bounded variance,
that is $\Yinput_{p,q,K}$ with \mbox{$q \geq 2$}.
We consider the worst case complexity when $\Yinput_{p,q,K}$ is intersected
with a ball of radius~\mbox{$\sigma > 0$}
with respect to the \emph{centered} $L_2$-norm
(i.e.~the standard deviation, which is a semi-norm),
\begin{equation*}
  \sigma\Barrel_2 := \{Y \in L_2 \mid \|Y - \expect Y\|_2 \leq \sigma\} \,.
\end{equation*}
In this case finding worst case lower bounds
reduces to distinguishing two Bernoulli random variables with antithetic
probability of success.

\begin{thm}[Worst case complexity for bounded variance]
  \label{thm:LB-YcapB2}
  Let \mbox{$1 \leq p < q \leq \infty$} with~$q \geq 2$ and 
  let $K>1$ as well as $\sigma>0$. Then
  \begin{equation*}
    \bar{n}^{\wor}(\eps,\delta, \Yinput_{p,q,K} \cap \sigma\Barrel_2)
      \,\geq\, \frac{1}{4 \log 3} \, \left(\frac{\sigma}{\eps}\right)^2
                 \, \log\left(\fracts{4}{3}\delta\right)^{-1}
  \end{equation*}
  for any $0 < \eps \leq \min\{1 - 2/(K+1),\, 1/2\}\, \sigma$,
  and~$0 < \delta \leq 1/4$.
\end{thm}
\begin{proof}
  We consider shifted and scaled Bernoulli random variables~$Y^{(1)}$, $Y^{(2)}$
  that only take values~$\pm \sigma$ with positive probability.
  For $i=1,2$, obviously, $Y^{(i)} \in \sigma\Barrel_2$
  since
  \begin{equation*}
    \|Y^{(i)} - \expect Y^{(i)}\|_2 \leq \|Y^{(i)}\|_2 = \sigma \,.
  \end{equation*}
  Further, by the triangle inequality we have
  \begin{equation*}
    \frac{\|Y^{(i)} - \expect Y^{(i)}\|_q}{\|Y^{(i)} - \expect Y^{(i)}\|_p}
      \,\leq\, \frac{\|Y^{(i)}\|_q + |\expect Y^{(i)}|
               }{\|Y^{(i)}\|_p - |\expect Y^{(i)}|}
      \,=\, \frac{2}{1 - |\expect Y^{(i)}|/\sigma} - 1 \,.
  \end{equation*}
  Hence, $|\expect Y^{(i)}| \leq (1 - 2/(K+1)) \, \sigma$ ensures
  $Y^{(i)} \in \Yinput_{p,q,K}$.
  Introducing a parameter \mbox{$\alpha\in (0,1 - 2/(K+1)]$},
  we further specify $Y^{(1)}$ and $Y^{(2)}$ by
  \begin{align*}
    \P\{Y^{(1)} = +\sigma\} \,=\, \P\{Y^{(2)} = -\sigma\}
      &\,=\, \frac{1 + \alpha}{2} \,,\\
    \P\{Y^{(1)} = -\sigma\} \,=\, \P\{Y^{(2)} = +\sigma\}
      &\,=\, \frac{1 - \alpha}{2} \,.
  \end{align*}
        Hence
        \begin{equation*}
    \expect Y^{(1)} \,=\, -\expect Y^{(2)} \,=\, \alpha \, \sigma \,.
  \end{equation*}
  Now let $\eps\in(0,\alpha \sigma)$ and note that
  from any $(\eps,\delta)$-approximating algorithm~$A$
  on~$\Yinput_{p,q,K} \cap \sigma \Barrel_2$
  we can build a statistical test~$T$
  distinguishing~$Y^{(1)}$ from $Y^{(2)}$ by
  \begin{equation*}
   T(Y^{(i)}) := \begin{cases}
                 1 & A(Y^{(i)})>0 \,,\\
                 2 & A(Y^{(i)})\leq 0\,.
                \end{cases}
  \end{equation*}
  Obviously, the expected cardinality of $A$ and $T$ is the same
  and the failure probability does not exceed~$\delta < \frac{1}{2}$.
  Then, from Proposition~\ref{prop:Wald-probrat} we obtain
  \begin{align*}
    \bar{n}(A,Y^{(i)})
      &\,\geq\, \frac{(1-2\delta) \log \frac{1-\delta}{\delta}
        }{\alpha \log \frac{1+\alpha}{1-\alpha}} \,.
  \end{align*}
  Restricting to~$0 < \alpha \leq 1/2$ we may exploit
  convexity of~$\log \frac{1 + \alpha}{1-\alpha}$,
  replacing it by $2 (\log 3) \, \alpha$.
  We also restrict to $0 < \delta \leq 1/4$ and obtain
  \begin{align*}
    \bar{n}(A,Y^{(i)})
      &\,\geq\, \frac{1}{4 \, \log 3} \,
          \, \alpha^{-2}
          \, \log \left(\fracts{4}{3}\delta\right)^{-1} \,.
  \end{align*}
  The limit~$\alpha \to \eps/\sigma$ and continuity in $\eps$
  gives the desired complexity bound
  for the range of error thresholds
  \mbox{$0 < \eps \leq \min\{1-2/(K+1),\, 1/2\} \, \sigma$}.
\end{proof}

For classes~$\Yinput_{p,q,K}$ with less integrability, $1 \leq p < q \leq 2$,
we intersect the cone with
a ball of radius $\tau > 0$ with respect to the centered $L_q$-norm,
\begin{equation*}
  \tau \Barrel_q := \{Y \in L_q \mid \|Y - \expect Y\|_q \leq \tau \} \,,
\end{equation*}
and we obtain the following worst case lower bound.

\begin{thm}[Worst case complexity for bounded central $L_q$-norm,
    $1 < q \leq 2$]
   %at bounded scale, low integrability] 
  \label{thm:LB-YcapBq}
  Let \mbox{$1 \leq p < q \leq 2$} and $K>1$ as well as $\tau>0$.
  Then there is a constant~$c_{q,K} > 0$, depending only on $K$ and $q$,
  such that
  \begin{equation*}
  \bar{n}^{\wor}(\eps,\delta, \Yinput_{p,q,K} \cap \tau\Barrel_q)
  \,\geq\, c_{q,K} \, \left( \frac{\tau}{\eps}\right)^{1 + 1/(q-1)}
      \, \log \left(\fracts{4}{3} \delta\right)^{-1}
  \end{equation*}
  for any $0 < \eps \leq \fracts{1}{6} \, (1-1/K) \, \tau$
  and $0 < \delta \leq 1/4$.
\end{thm}
\begin{proof}
  We consider random variables~$Y^{(1)}$, $Y^{(2)}$ with   
  \begin{align*}
    \P\{Y^{(i)} = \tau'\}
      \,=\, \P\{Y^{(i)} = -\tau'\}
      &\,=\, \frac{1 - \beta}{2} \,, \\
    \P\{Y^{(i)} = 0\} &\,=\, \beta  \, (1 - \gamma) \,, \\
%     \quad \text{and}\\
    \P\{|Y^{(i)}| = \gamma^{-1/q} \, \tau'\} &\,=\, \beta \, \gamma \,,
  \end{align*}
  for $i=1,2$,
  where~$0 < \tau' < \tau$ and $0 < \beta,\gamma < 1$.
  Now we choose $\beta$ and $\tau$
  such that $Y^{(i)}\in \Yinput_{p,q,K}\cap\tau\Barrel_q$.
  Observe that
  \begin{gather*}
   |\expect Y^{(i)}| \leq \beta \, \gamma^{1-1/q} \, \tau'
    \leq \beta \, \tau',\\
   (1-\beta) \, \tau' 
      \leq (1-\beta)^{1/p} \, \tau' \leq \|Y^{(i)}\|_p
      \leq \|Y^{(i)}\|_q = \tau'.
  \end{gather*}
  Using the triangle inequality,
  for $0 < \beta < 1/2$ we obtain 
  \begin{equation*}
    \frac{\|Y^{(i)} - \expect Y^{(i)}\|_q}{\|Y^{(i)} - \expect Y^{(i)}\|_p}
      \,\leq\, \frac{\|Y^{(i)}\|_q + |\expect Y^{(i)}|
            }{\|Y^{(i)}\|_p - |\expect Y^{(i)}|}
      \,\leq\, \frac{1 + \beta}{1 - 2\beta},
  \end{equation*}
  such that if we put
  \begin{equation*}
    \beta := \frac{1}{2} \, \left(1 - \frac{3}{1 + 2K} \right) \,,
  \end{equation*}
  this ensures~$Y^{(i)} \in \Yinput_{p,q,K}$.
  Furthermore,
  \begin{equation*}
    \|Y^{(i)} - \expect Y^{(i)}\|_q
      \,\leq\, \|Y^{(i)}\|_q + |\expect Y^{(i)}|
      \,\leq\, (1 + \beta) \, \tau' \,,
  \end{equation*}
  so choosing $\tau' := \tau/(1+\beta)$
  guarantees $Y^{(i)} \in \tau\Barrel_q$.
  We %further
  specify $Y^{(1)}$ and $Y^{(2)}$ by
  \begin{align*}
    \P\{Y^{(1)} = \gamma^{-1/q} \tau'\}
      \,=\, \P\{Y^{(2)} = - \gamma^{-1/q} \tau'\}
      \,&=\, \frac{3}{4} \, \beta \, \gamma \,,\\
    \P\{Y^{(1)} = - \gamma^{-1/q} \tau'\}
      \,=\, \P\{Y^{(2)} = \gamma^{-1/q} \tau'\}
      \,&=\, \frac{1}{4} \, \beta \, \gamma \,.
  \end{align*}
%  Hence,
  \checked{Their respective expected values are then given by}
  %Robert: Das war die Aufgabe ``additional wording''
  \begin{equation*}
    \expect Y^{(1)}
      \,=\, - \expect Y^{(2)}
      \,=\,\frac{1}{2}\,\beta\, \gamma^{1-1/q} \, \tau' \,.
  \end{equation*}
  In order to guarantee that
  $\frac{1}{2} \beta \, \gamma^{1-1/q} \, \tau' > \eps$,
  we choose $\gamma>\left(2\eps/(\beta \tau')\right)^{1+1/(q-1)}$.
  Since we need \mbox{$\gamma<1$},
  this is only possible for $\eps\in(0,\frac{1}{2}\beta\tau')$.
  As in the proof of Theorem~\ref{thm:LB-YcapB2},
  cost bounds for $(\eps,\delta)$-approximating algorithms~$A$
  on $\Yinput_{p,q,K}\cap \tau\Barrel_q$
  are derived from the cost bounds of Proposition~\ref{prop:Wald-probrat}
  for distinguishing~$Y^{(1)}$ from~$Y^{(2)}$,
  \begin{equation*}
    \bar{n}(A,Y^{(1)})
      \,\geq\, \frac{2(1-2\delta) \log \frac{1-\delta}{\delta}
                }{\beta\, \gamma\, 
                \log3 } \,.
  \end{equation*}
  By letting
  $\gamma\to \left(2\eps(1+\beta)/(\beta \tau)\right)^{1+1/(q-1)}$,
  we obtain
  \begin{equation*}
    \bar{n}^{\wor}(\eps,\delta,\Yinput_{p,q,K} \cap \tau\Barrel_q)
      \,\geq\, 
        c_{q,K} \, \left(\frac{\tau}{\eps}\right)^{1+1/(q-1)}
                \, \log \left(\fracts{4}{3} \delta\right)^{-1}
  \end{equation*}
  with $%\displaystyle
        c_{q,K}:=\frac{1}{\beta \, \log 3}
          \, \left(\frac{\beta}{2(1+\beta)}\right)^{1+1/(q-1)}$,
  and the assertion is proven.
\end{proof}

\checked{
\begin{rem}[Dependence on $K$] \label{rem:worLB_K}
  The value of the lower bound in Theorem~\ref{thm:LB-YcapB2} does not depend on~$K$,
  and in Theorem~\ref{thm:LB-YcapBq} the prefactor~$c_{q,K}$
  approaches a \mbox{$q$-dependent} constant as~$K$ grows to infinity.
  In contrast, our upper bounds in Theorem~\ref{thm:costbound}
  do have a $K$-dependence,
  but as pointed out in Remark~\ref{rem:balance},
  in cases of cones~$\Yinput_{p,q,K}$ with $q > 2$
  we can directly estimate the variance (which is the moment of interest)
  and no $K$-dependence occurs in the scale dependent part of the cost bound,
  that is, the part which depends on the norm of the input and the accuracy~$\eps$.
  In that respect, the lower bounds seem to be optimal.
  
  As $K > 1$ approaches~$1$, though,
  the lower bounds only hold for small $\eps > 0$.
  This is not reflected in the upper bound of Theorem~\ref{thm:costbound},
  likely because when constructing our algorithm we ignored the fact
  that the set~\mbox{$\Yinput_{p,q,K} \cap \tau\Barrel_{\tilde{q}}$}
  might be much smaller than the set~\mbox{$\tau\Barrel_{\tilde{q}}$},
  for~$\tau > 0$ and $\tilde{q} := \min\{q,2\}$.
  Indeed, for small~$K$
  the distributions of random variables contained in~$\Yinput_{p,q,K}$
  tend to be bimodal, 
  so observing samples from both modes
  could already enable for a rough approximation of the expectation
  by returning a value right between the two modes.
  To get a better intuition we consider the degenerate case~$K = 1$,
  where $\Yinput_{p,q,K}$, independently from $p$ and $q$,
  contains only random variables that either take one constant value,
  or two distinct values with probability~$1/2$ each.
  Then, for any~$\delta \in (0,1)$ there is an i.i.d.-based algorithm
  that provides an exact solution with probability at least~\mbox{$(1-\delta)$}.
  In detail, put~\mbox{$n := \lceil \log_2 \delta^{-1} \rceil + 1$}
  and define
  \begin{equation*}
  A(Y) := \fracts{1}{2} \left(\min\{Y_1,\ldots,Y_n\}
  + \max\{Y_1,\ldots,Y_n\}
  \right) \,.
  \end{equation*}
  If~$Y$ has two distinct values,
  it is with probability~$2^{-(n-1)}$ that we only observe one of these.
  If, however, we observe both values, the output will perfectly match
  the expected value.
  Note that in this setting
  the fixed cost lower bound of Theorem~\ref{thm:fixedcost} still applies.
  
  The prefactor~$c_{q,K}$ in Theorem~\ref{thm:LB-YcapBq}
  tends to zero for small~$K$.
  We suspect that this is a shortcoming of the lower bound proof.
\end{rem}
}

\subsection{Extension to general integration methods}
\label{sec:INTLBs}

Here we consider the integration problem defined in~\eqref{eq:INT}.
Intuitively, as long as the measure $\pi$ has no atoms on $G$,
it seems reasonable to restrict to i.i.d.-based algorithms.
(If, however, $\pi$ has atoms, then there is a non-zero probability 
that we compute the same function value twice when using i.i.d.\ samples.)
We are going to show that the previously obtained lower bounds
for i.i.d.-based algorithms also hold for general randomized algorithms.
The key observation is that $L_p$-norms do not reflect any kind of
local structure of functions.

We impose the following weak standing assumption concerning the 
existence of suitable partitions, which for $G\subseteq \R^d$,
equipped with the Borel $\sigma$-algebra, is always satisfied whenever $\pi$
has no atoms.
\begin{ass}
  For every $N\in \N$ we assume that there are 
  pairwise disjoint
  measurable sets $G_{1,N},\ldots,G_{N,N} \subset G$
%  $G_{1,N},\ldots,G_{N,N}\in \Sigma_G$
  with $\bigcup_{j=1}^N G_{j,N}=G$ 
%  satisfying
  possessing equal measure
  \mbox{$\pi(G_{j,N})=1/N$} for all \mbox{$j=1,\ldots,N$}.  
\end{ass}

While lower bounds for i.i.d.~based algorithms
were obtained by lower bounds for distinguishing
random variables~$Y^{(1)}$ and $Y^{(2)}$ with expected values
at a distance exceeding $2 \eps$,
now the idea is to distinguish between two collections~$\Finput^{(1)}$
and $\Finput^{(2)}$ of functions
with distinct integral values,
\begin{equation*}
  \Int(f)
    = \begin{cases}
        a_1 &\quad\text{if $f \in \Finput^{(1)}$,}\\
        a_2 &\quad\text{if $f \in \Finput^{(2)}$,}
      \end{cases}
\end{equation*}
at distance $|a_2 - a_1| > 2 \eps$.
Such collections are constructed as follows.

We consider only functions that take at most
finitely many
distinct values,
each with a certain probability.
For this let $D\subset \R$ with $\# D < \infty$ %$\# D = k \in \N$
and
\begin{equation*}
  p^{(i)}_z \in [0,1]\cap\mathbb{Q}\,, 
  \qquad \text{for $z\in D$ and $i=1,2$,}
\end{equation*}
satisfying 
\begin{equation*}
\sum_{z\in D} p^{(i)}_z = 1\,, \qquad \text{for $i=1,2$}.
\end{equation*}
Consider the (unbounded) set of all common denominators of these probabilities,
namely,
\begin{equation*}
  \D_{\vecp}
    := \{N \in \N \mid N \cdot p_{z}^{(i)} \in \N_0
    \text{\; for all 
%     $\ell = 1,\ldots,k$ 
    $z\in D$
    and $i=1,2$}\} \,.
\end{equation*}
Then, for $N \in \D_{\vecp}$ and $i=1,2$, the sets
\begin{align*}
  \Finput_N^{(i)}
    := \bigl\{f: G \to \R \mid&\text{$f$ is constant on each $G_{j,N}$,
           where $j=1,\ldots,N$;}\\
        &\text{and $\pi\{\vecx \colon f(\vecx) = z\}
        = p_z^{(i)}$
%     for $\ell = 1,\ldots,k$
    for $z\in D$
    }
        \bigr\} 
\end{align*}
are non-empty.
In detail, for $f \in \Finput_N^{(i)}$,
the number of subdomains~$G_{j,N}$ with $f(G_{j,N}) = z$
is exactly \mbox{$N \cdot p_{z}^{(i)} \in \N_0$}.
Furthermore, for any function~$f \in \Finput_N^{(i)}$
and any permutation~$\Pi$ on~$\{1,\ldots,N\}$,
the function~$f_{\Pi}$ with $f_{\Pi}(G_{j,N}) := f(G_{\Pi(j),N})$
is also contained in~$\Finput_N^{(i)}$.
Now, for $i=1,2$ we define
\begin{equation*}
  \Finput^{(i)} := \bigcup_{N \in \D_{\vecp}} \Finput_N^{(i)}. 
\end{equation*}
This leads to the integral value
\begin{equation*}
\Int (f) = \sum_{z \in D} p_z^{(i)} z =: a_i
\end{equation*}
for $f \in \Finput^{(i)}$. 

\begin{prop} \label{prop:iid=general}
  Let $\vecX$ be a $\pi$-distributed random variable 
  and let $\Finput^{(i)}$ with $i=1,2$ be defined as above.
  Then
  \begin{equation*}
    \inf_Q \max_{i=1,2} \sup_{f \in \Finput^{(i)}} \P\{Q(f) \not= a_i\} \\
      \,=\,
  \inf_A \max_{i=1,2}\sup_{f \in \Finput^{(i)}} \P\{A(f(\vecX)) \not= a_i\} \,,
  \end{equation*}
  where the infimum on the left-hand side 
  runs over all general randomized algorithms~$Q$ with
  \begin{equation} \label{eq:n_i-constraint}
    \sup_{f \in \Finput^{(i)}} \bar{n}(Q,f) \leq \bar{n}_i
    \qquad \text{for $i=1,2$\,,}
  \end{equation}
  and on the right-hand side 
  over all i.i.d.-based algorithms~$A$ with the same cost constraint.
  In particular, this implies that 
  no general randomized algorithm is better than
   the best possible i.i.d.-based algorithm
   in distinguishing $\Finput^{(1)}$ from $\Finput^{(2)}$.
\end{prop}
\begin{proof}
  For $N \in \D_{\vecp}$ let~$\mu_N^{(i)}$ be the uniform distribution
  on~$\Finput_N^{(i)}$, for $i=1,2$.
  Then, for any general randomized algorithm~$Q$
  defined on $\Finput^{(1)} \cup \Finput^{(2)}$
  we have
  \begin{equation} \label{eq:n_i,mu}
    \bar{n}_i
      \geq \sup_{f \in \Finput^{(i)}} \bar{n}(Q,f)
      \geq \int_{\Finput_N^{(i)}} \bar{n}(Q,f) \dint \mu_N^{(i)}(f) \,,
  \end{equation}
  and
  \begin{align}
    \sup_{f \in \Finput^{(i)}} \P\{Q(f) \not= a_i\}
      &\geq \int_{\Finput_N^{(i)}} \P\{Q(f) \not= a_i\}
              \dint\mu_N^{(i)}(f)
        \nonumber\\
      &\stackrel{\text{Fubini}}{=}
        (\P\otimes\mu_N^{(i)})\{(\omega,f) \colon Q(\omega,f) \not= a_i\} \,.
      \label{eq:murisk}
  \end{align}
  
  We now consider general randomized algorithms~$Q$
  that minimize the maximal average uncertainty
  \begin{equation} \label{eq:crit-maxavg}
    \max_{i=1,2}\,
      (\P\otimes\mu_N^{(i)})\{(\omega,f) \colon Q(\omega,f) \not= a_i\}
  \end{equation}
  while satisfying the constraint~\eqref{eq:n_i-constraint}.
  (Replacing a randomized worst case setting by an average case setting
    for lower bound proofs
    is a trick that has already been used by Bakhvalov 1959~\cite{Bakh59}.)
%   Optimal methods with respect to this average criterion compute
%   no more than one function value for each of the~$G_{j,N}$.
%   Without loss of generality, due to the symmetry of~$\mu_N^{(i)}$,
%   we say that those methods compute function values~$f(\vecx_j)$
%   with \mbox{$\vecx_j \in G_{j,N}$},
%   where~$j=1,\ldots,N$.
\checked{When using $n$ function evaluations,
    the function is known on at most $n$~of the subdomains of $G_{1,N},\ldots,G_{N,N}$.
    Without loss of generality,
    we may assume that the algorithm indeed gains knowledge on $n$~different subdomains.
    (If not, then a modified method with improved information mapping
    that avoids computing two function values on the same subdomain
    while preserving the knowledge of the original information map,
    can still return the same output or even improve the algorithm.)
    Due to symmetry of the probability measure $\mu_N^{(i)}$,
    the \checked{conditional} distribution of function values
    on the remaining $N-n$ subdomains
    is independent of the choice of the $n$ subdomains where 
    %we computed function values.
    \checked{the function has been evaluated.}
    Hence, for simplicity, we assume that the algorithm uses one function value
    within each of the first $n$ subdomains $G_{1,N},\ldots,G_{n,N}$.}
  According to Definition~\ref{def:alg},
  such a method can be described via a random variable~$U$,
  a stopping rule~\mbox{$n(u,y_1,y_2,\ldots)$},
  and an output function~$\phi(u,y_1,y_2,\ldots)$,
  such that
  \begin{equation*}
    Q(\omega,f)
      = \phi(U(\omega),f(\vecx_1),\ldots,f(\vecx_{n(\omega,f)}),0,0,\ldots)
    \,,
  \end{equation*}
  with stopping time given by
  $n(\omega,f) := n(U(\omega),f(\vecx_1),\ldots,f(\vecx_N),0,0,\ldots)$,
  \checked{and $\vecx_j \in G_{j,N}$ for $j=1,\ldots,N$.}
  In other words, the infimum over all general randomized algorithms~$Q$
  for the quantity~$\eqref{eq:crit-maxavg}$
  can be reduced to an infimum over
  stopping rules $n(u,y_1,y_2,\ldots)$ and
  output functions~$\phi$
  while the information of function values is predetermined.
  
  These algorithmic components $n$ and
  $\phi$ can also be used with an i.i.d.\ sample
  $(f(\vecX_j))_{j\in\N}$,
  where \mbox{$(\vecX_j)_{j\in \N}$} is a sequence of independent,
  $\pi$-distributed random variables that are independent of~$U$.
  By this we obtain an i.i.d.-based algorithm %~$A$ with
  \begin{equation*}
    A(\omega,f(\vecX))
      := \phi(U,f(\vecX_1(\omega)),\ldots,f(\vecX_{n(\omega,f(\vecX))}(\omega)),0,0,\ldots),
  \end{equation*}
  where
  $n(\omega,f(\vecX))
    := n(U(\omega),f(\vecX_1(\omega)),f(\vecX_2(\omega)),\ldots)$.
  For any such~$A$ we have
  \begin{equation*}
    (\P\otimes\mu_N^{(i)})\{(\omega,f): A(\omega,f(\vecX)) \not= a_i\}
      = \sup_{f \in \Finput^{(i)}} \P\{A(f(\vecX)) \not= a_i\}\,,
  \end{equation*}
  since $\P\{A(f(\vecX))\not=a_i\}= \P\{A(g(\vecX))\not=a_i\}$
  for all \checked{representing functions} \mbox{$f,g\in \Finput^{(i)}$}
  \checked{of the same subclass}.
  
  For the following considerations we restrict to methods
  with some absolute cardinality limit~$n(\checked{\omega},f)\leq M \leq N$.
  If a general randomized algorithm~$Q$ occasionally uses more than~$M$ samples,
  by Markov's inequality,
  \checked{in the $\mu_N^{(i)}$-average setting}
  this happens only with \checked{average} probability
  \begin{equation} \label{eq:Mrisk}
    (\P\otimes\mu_N^{(i)})\{(\omega,f) \mid n(\omega,f) > M\}
      \leq \frac{\bar{n}_i}{M} \,.
%    \qquad \text{for $f \in \Finput^{(i)}$.}
%% Robert: Das hat keinen Sinn gemacht.
  \end{equation}
  We can define another general randomized algorithm~$Q'$ 
  based on~$Q$ by simply stopping~$Q$
  whenever it would compute yet another function value
  after already knowing~$M$ function values,
  and returning an arbitrary output then.
  The uncertainty %only
  increases
  \checked{at most}
  by~$\bar{n}_i/M$
  (and vanishes in the limit $M \to \infty$ which we perform later on),
  whereas the cost is even reduced.
  
  Now we compare the probabilities at which function values are observed
  by a general randomized algorithm~$Q_N$ in the $\P\otimes\mu_N^{(i)}$ setting
  and by an i.i.d.-based algorithm~$A$,
  where both methods are constructed from the same
  stopping rule~$n$ and output function~$\phi$.
  (We write~$Q_N$ instead of~$Q$ to make clear
  that the method is based on the previously described optimal information
  \checked{with function evaluations at nodes $\vecx_k \in G_{k,N}$}
  for the~$\mu_N^{(i)}$-setting
  while \checked{the stopping rule} $n(u,y_1,y_2,\ldots)$
  and \checked{the output function} $\phi$ remain unchanged.)
  If $n(\omega,f)$ is absolutely bounded by~$M \leq N$,
  it suffices to consider the probability of different outcomes
  of~$M$ function evaluations.
  For an i.i.d.\ sample we have
  \begin{equation} \label{eq:f(X_j)=z_l,iid}
    \P\{f(\vecX_1) = z_{1},\ldots,f(\vecX_M) = z_{M}\}
      = \prod_{j=1}^M p_{z_j}^{(i)}
    \qquad \text{for $f \in \Finput^{(i)}$
     and $z_1,\ldots,z_M\in D$,
    }
  \end{equation}
  and in the $\mu_N^{(i)}$-average setting we have
  \begin{align}
    \prod_{j=1}^M \frac{N \cdot p_{z_j}^{(i)} - M}{N}
      &\leq \mu_N^{(i)}\{f(\vecx_1) = z_{1},\ldots,f(\vecx_M) = z_{M}\} 
      \nonumber\\
      &= \prod_{j=1}^M \frac{N \cdot p_{z_j}^{(i)}
        - \#\{\ell=1,\ldots,j-1 \colon z_{\ell} = z_j\}
              }{N-(j-1)}
      \label{eq:f(x_j)=z_l,Q}\\
      &\leq \prod_{j=1}^M \frac{N \cdot p_{z_j}^{(i)}}{N-M} \,,
      \nonumber
  \end{align}
  where the lower bound holds if~$M \leq N \cdot p_{z}^{(i)}$
  for all~$z\in D$ and $i=1,2$.
  This is true for large~$N$,
  and in the limit of $N \to \infty$ while $M$ is fixed,
  the probability of~\eqref{eq:f(x_j)=z_l,Q} 
  approaches~\eqref{eq:f(X_j)=z_l,iid}.
  Determining the uncertainty probability is now a discrete problem,
  namely, by the independence of~$U$ from the i.i.d.-sample we have
  \begin{multline*}
  \P\{A(f(\vecX)) \not= a_i\} = \\
    \sum_{(z_1,\ldots,z_M)\in D^M}
      \P\{\phi(U,z_{1},\ldots,z_{M},0,\ldots) \not= a_i\}
      \cdot \P\{f(\vecX_1) = z_{1},\ldots,f(\vecX_M) = z_{M}\} 
  \end{multline*}
  for $f \in \Finput^{(i)}$.  
  On the other hand, within the~$\mu_N^{(i)}$-average setting we have
  \begin{multline*}
    (\P\otimes\mu_N^{(i)})\{(\omega,f) \colon Q_N(\omega,f) \not= a_i\} =  \\
      \sum_{(z_1,\ldots,z_M)\in D^M}
        \P\{\phi(U,z_{1},\ldots,z_{M},0,\ldots) \not= a_i\}
        \cdot \mu_N^{(i)}\{f(\vecx_1) = z_{1},\ldots,f(\vecx_M) = z_{M}\} \,.
  \end{multline*}
  Here, we wrote \mbox{$\phi(U,z_1,\ldots,z_M,0,\ldots)$},
    which %of course
    in fact
    only depend on \mbox{$z_1,\ldots,z_{n(U,\vecz)}$}
    where \mbox{$n(U,\vecz) := n(U,z_1,\ldots,z_M,0,\ldots)$}.
  Hence, for any $\eta > 0$ there is an \mbox{$N_0 \in \D_{\vecp}$}
  such that for all \mbox{$N \geq N_0 \geq M$},
  \checked{for all stopping rules 
  $n(u,y_1,y_2,\ldots) \leq M$,
  and for all} output functions~$\phi$,
  the corresponding methods~$Q_N$ and $A$ satisfy for $f \in \Finput^{(i)}$
  that
  \begin{equation} \label{eq:etarisk}
    \P\{A(f(\vecX)) \not= a_i\}
      \,\leq\, \eta
          + (\P\otimes\mu_N^{(i)})\{(\omega,f) \colon Q_N(\omega,f) \not= a_i\}
%    \qquad \text{for $f \in \Finput^{(i)}$.}
            \,.
  \end{equation}
  Similarly, for any~$\nu > 0$
  and sufficiently large~$N$ we obtain
  \begin{align*}
    \bar{n}(A,&f(\vecX))
      = \sum_{(z_1,\ldots,z_M)\in D^M}
          \expect n(U,z_{1},\ldots,z_{M},0,\ldots)
             \cdot %\P\{f(\vecX_1) = z_{1},\ldots,f(\vecX_M) = z_{M}\}
             %% Robert: Aus Platzgründen
                   \P\{f(\vecX_i) = z_i,\,i=1,\ldots,M\} \\
      &\,\leq\,
        \nu
        + \sum_{(z_1,\ldots,z_M)\in D^M}
            \expect n(U,z_{1},\ldots,z_{M},0,\ldots)
             \cdot %\mu_N^{(i)}\{f(\vecx_1) = z_{1},\ldots,f(\vecx_M) = z_{M}\}
                   \mu_N^{(i)}\{f(\vecx_i) = z_i,\,i=1,\ldots,M\} \\
      &\,=\,
        \nu + \int_{\Finput_N^{(i)}} \bar{n}(Q_N,f) \dint \mu_N^{(i)}(f)
      \stackrel{\eqref{eq:n_i,mu}}{\,\leq\,} \nu + \bar{n}_i \,.
  \end{align*}
  Hence, $A$ does not satisfy the constraint~\eqref{eq:n_i-constraint},
  but~$\nu\to 0$ as $N \to \infty$.
  The violation of the constraint
  can be fixed by defining a modified method~$A'$ that,
  independently from all random variables used in~$A$,
  decides with probability
  \begin{equation} \label{eq:zetarisk}
    \zeta := \max_{i=1,2} \frac{\nu}{\nu + \bar{n}_i}
  \end{equation}
  that no information at all shall be collected
  and simply returns an arbitrary value,
  whereas with probability~$1-\zeta$ the original algorithm $A$ is executed.
  That way the uncertainty of the method~$A'$ only rises by~$\zeta$,
  which vanishes in the limit~$\nu \to 0$.
  
  To summarize, since i.i.d.-based algorithms are special general randomized algorithms,
  for sufficiently large~$N$ we have
  \begin{multline*}
    \inf_A \max_{i=1,2}
      \sup_{f \in \Finput^{(i)}}
       \P\{A(f(\vecX)) \not= a_i\}
      \stackrel{\eqref{eq:murisk}}{\geq}
        \inf_Q \max_{i=1,2} \,
          (\P\otimes\mu_N^{(i)})\{(\omega,f) \colon Q(\omega,f) \not= a_i\} \\
    \stackrel{\eqref{eq:Mrisk},\eqref{eq:etarisk},\eqref{eq:zetarisk}}{\geq}
        \inf_A \max_{i=1,2}
            \sup_{f \in \Finput^{(i)}}
          \P\{A(f(\vecX)) \not= a_i\}
          \checked{ - \frac{\bar{n}_i}{M} - \eta - \zeta}
          \,,
  \end{multline*}
  where the infimum runs over all methods
  satisfying the cardinality constraint~\eqref{eq:n_i-constraint}.
  Fixing~$M$ and letting $N\to\infty$, $\eta$ and $\zeta$ vanish.
  Finally, the limit $M\to\infty$ proves the assertion.
\end{proof}

An immediate consequence of the previous
proposition is that the different types of 
lower bounds for i.i.d.-based algorithms on
$\Yinput_{p,q,K}$ carry over to 
lower bounds for general randomized algorithms on $\Finput_{p,q,K}(\pi)$.

\begin{cor} \label{cor:LBsINT}
Let $1\leq p<q\leq \infty$, $K>1$ and $\tau>0$.
Then, for the integration problem \eqref{eq:INT},
\begin{enumerate}
 \item on~$\Finput = L_q(\pi)$ we obtain unsolvability, 
   see corresponding Theorem~\ref{thm:Lq--noway};
\end{enumerate}
and on cone-shaped input classes $\Finput_{p,q,K}(\pi)$ we obtain
\begin{enumerate} \setcounter{enumi}{1}
 \item impossibility for any general randomized algorithm $Q$
   to guarantee a finite worst case expected error, that is, 
   \begin{equation*}
     \sup_{f\in \Finput_{p,q,K}(\pi)} \expect | Q(f)-\Int(f) |
       \,=\, \infty,
   \end{equation*}
   see corresponding Theorem~\ref{thm:E(err)<eps--noway};
 \item the following lower bound for the fixed cost
   of general randomized algorithms,
   \begin{equation*}
     \bar{n}^{\fix}(\eps,\delta,\Finput_{p,q,K}(\pi))
       \,\geq\, \frac{1}{2\log 2} \, K^{pq/(q-p)} \, \log \delta^{-1} 
   \end{equation*}
   for any $\eps > 0$ and $\delta \in (0,1)$,
   see corresponding Theorem~\ref{thm:fixedcost};
 \item 
    the following lower bound for the worst case complexity
    for bounded norm,
    \begin{equation*}
      \bar{n}^{\wor}(\eps,\delta,\Finput_{p,q,K}(\pi)\cap \tau\Barrel_{\wt q})
        \,\geq\,
          c_{q,K} \, \left(\frac{\tau}{\eps}\right)^{1+1/(\wt{q}-1)}
            \, \log\left(\fracts{4}{3}\delta\right)^{-1}
   \end{equation*}
   for any $0<\eps< \frac{1}{6}\,(1-1/K)\,\tau$ and $0 < \delta < 1/4$,
   with $\wt{q}:=\min\{q,2\}$,
   and a constant $c_{q,K}>0$ depending only on $q$ and $K$,
   see corresponding Theorems~\ref{thm:LB-YcapB2} and \ref{thm:LB-YcapBq}.
\end{enumerate}
\end{cor}
\begin{proof}
  In all cases, the random variables have to be modified
  \checked{before building corresponding function classes},
  now possessing rational weights for the different values
  such that the cone condition within~$\Yinput_{p,q,K}$ is still satisfied.
  Since~$\Q$ is dense in~$\R$, we obtain the same lower bounds.
  Cost bounds at given confidence level (as %given
  to be found in the cited theorems)
  are equivalent to uncertainty bounds at given cost budget
  (as considered in Proposition~\ref{prop:iid=general}),
  namely due to monotonicity and continuity
  of the connection between the two quantities,
% simple
  inversion does the job.
  
  In Theorems~\ref{thm:LB-YcapB2} and~\ref{thm:LB-YcapBq}
  we are interested in the worst case cost,
  so the cost constraint we impose
  when applying Proposition~\ref{prop:iid=general}
  will be with~$\bar{n}_1 = \bar{n}_2$.
  
  For generalizing the unsolvability results Theorems~\ref{thm:Lq--noway}
  and~\ref{thm:E(err)<eps--noway},
  the cardinality constraint~\eqref{eq:n_i-constraint}
  in Proposition~\ref{prop:iid=general}
  must be replaced by a constraint of type
  \begin{equation}
    \P\{\bar{n}(A,f) > M \} \leq  \nu_M
     \qquad \text{for all $f \in \Finput^{(1)} \cup \Finput^{(2)}$,}
  \end{equation}
  with a monotonously decaying sequence~$(\nu_M)_{M \in \N_0}$
  converging to~$0$.
  Consequently, within the proof of Proposition~\ref{prop:iid=general},
  the estimate~\eqref{eq:Mrisk} has to be adjusted accordingly,
  but all following arguments with the limit~$M\to\infty$ will work the same.
  In detail, $\Finput^{(1)}$ only consists of the zero function.
  Since an algorithm that shall distinguish the zero function
  from a class~$\Finput^{(2)}$ consisting of non-zero functions
  can stop when observing a non-zero function value,
  we may assume~$\nu_M := \P\{\bar{n}(A,0) > M\}$,
  so the constraint is actually no restriction.
  
  For generalizing Theorem~\ref{thm:fixedcost}
  we need to extend Proposition~\ref{prop:iid=general}
  to the task of distinguishing finitely many function classes
  $\Finput^{(0)},\ldots,\Finput^{(m)}$
  with integral values at separated distances~$>2\eps$,
  the proof works similar.
  \checked{By taking the limit $m \to \infty$,}
  we may even consider countably many function 
  classes~$\Finput^{(i)}$, $i \in \N_0$.
  Now for the detailed transcription of
  the proof of Theorem~\ref{thm:fixedcost}:
  We take any function~$\checked{f^*} \in \Finput_{p,q,K}$
  %andere Bezeichnung, um Verwirrung zu reduzieren.
  (which corresponds to~$Y \in \Yinput_{p,q,K}$,
  the input for which we eventually derive a cost bound);
  further we consider function classes~$\Finput^{(i)}$ with \mbox{$i \in \N_0$}
  (which correspond to random variables of type~$c_i Z$),
  such that 
  $\checked{f := f^* + \wt{f}} \in \Finput_{p,q,K}$ for~$\checked{\wt{f}} \in \Finput^{(i)}$.
  The set $\Finput^{(0)}$ only contains the zero function,
  for all other indices~$i \in \N$ we \checked{specify} %have
  for $\wt{f} \in \Finput^{(i)}$ that
  \begin{equation*}
    \pi\{\wt{f}(\vecx) = 0\} = 1-a \,,
    \qquad \text{and}\qquad \pi\{\wt{f}(\vecx) = c_0 + 3\,i\eps/a\} = a \,,
  \end{equation*}
  with~$c_0 > 0$ and $a \in \Q$ as in the proof of Theorem~\ref{thm:fixedcost}.
  Consequently,
  \begin{equation*}
    \Int(\wt{f})
      = \begin{cases}
          0 &\quad\text{for $\wt{f} \in \Finput^{(0)}$,}\\
          a c_0 + 3\,i\eps
            &\quad\text{for $\wt{f} \in \Finput^{(i)}$
             with~$i \in \N$.}
        \end{cases}
  \end{equation*}
  The task of determining~$\Int(f) = \Int(f^* + \wt{f})$,
  not knowing the index $i \in \N_0$ for which \mbox{$\wt{f} \in \Finput^{(i)}$},
  can be traced back to determining~$\Int(\wt{f})$,
  since for any general algorithm~$Q$
  \checked{acting on~$\bigcup_{i = 0}^{\infty} \Finput^{(i)}$,}
  by
  \begin{equation*}
    Q'(f) := Q(f - f^*) \checked{+ \Int(f^*)}
  \end{equation*}
  another algorithm
  \checked{acting on $f^* + \bigcup_{i = 0}^{\infty} \Finput^{(i)}$
  and exhibiting the same reliability}
  is defined.
  This also holds for algorithms that use i.i.d.\ samples,
  because in the integration setting the points~$\vecX_i\sim\pi$
  are user-generated and one may use~$f^*(\vecX_i)$ within~$Q'$.
  Since %in this setting
  it suffices to observe one function value
  \mbox{$\wt{f}(\vecx_j) = z_i \not= 0$}
  in order to know that~$\wt{f} \in \Finput^{(i)}$ with~$i \not= 0$,
  the cost \mbox{$\bar{n}_0 = \bar{n}(Q,f^*) = \bar{n}^{\wor}(Q',\Finput^{(0)})$}
  is an upper bound for the worst case cost
  \mbox{$\bar{n}_i = \bar{n}^{\wor}(Q',\Finput^{(i)})$}
  in the other cases.
\end{proof}

% So far we did not mention an unsolvability result
% for the whole class of $q$-integrable functions, $\Finput = L_q(\pi)$,
% similar to Theorem~\ref{thm:Lq--noway}.
% This is because
We can formulate an even stronger unsolvability result
that holds for the small space of smooth functions. 
%E  Den letzten Satz habe ich gekuerzt. 

\begin{thm} \label{thm:Cinf}
  The integration problem~\eqref{eq:INT} for the class of infinitely
  differentiable functions,
  \begin{equation*}
    \Finput = C^{\infty}([0,1]) \,,
  \end{equation*}
  with respect to the Lebesgue measure on~$[0,1]$,
  is not solvable.
\end{thm}
\begin{proof}
  Split $[0,1]$ into~$N$ intervals of equal length,
  \mbox{$G_i := [(i-1)/N, i/N)$}, with \mbox{$i=1,\ldots,N$}.
  It is well known that there exist $C^{\infty}$-functions~$f_i$ on $\R$
  that are supported on the closures of the intervals~$\overline{G_i}$
  and positive on the interior of~$G_i$.
  Scaling does not affect the property
  of belonging to the class $\Finput = C^{\infty}([0,1])$,
  so~$\Int(f_i)$ can be arbitrarily large.
  Similarly to the proof of Proposition~\ref{prop:iid=general},
  we aim to distinguish the zero function from a $\mu^{(2)}_N$-average input
  where~$\mu^{(2)}_N$ is the uniform distribution on the set
  \mbox{$\Finput^{(2)}_N := \{f_1,\ldots,f_N\}$}.
  Analogously to the proof of Theorem~\ref{thm:Lq--noway},
  the $\mu^{(2)}_N$-probability of observing a non-zero function value
  after $n_0$~function evaluations vanishes in the limit $N\to\infty$.
  In the same manner we then obtain
  that for any $\eps > 0$ and $\delta \in (0,1/2)$
  there is no $(\eps,\delta)$-approximating algorithm
  for the whole class~$C^{\infty}([0,1])$.
\end{proof}

\section*{Acknowledgements}

We thank Fred Hickernell, Mark Huber, and Henryk Wo\'zniakowski 
for many valuable discussions over several years 
and we also thank the referees. 
Parts of the work have been done at the ESI in Vienna during the programme
``Tractability of High Dimensional Problems and Discrepancy''
in October 2017.
Daniel Rudolf gratefully acknowledges support of the 
Felix-Bernstein-Institute for Mathematical Statistics in the Biosciences
(Volks\-wagen Foundation) and the Campus laboratory AIMS.

%   \en{Mehr Bedankungen?} 
%E  Wir behalten das im Auge.


\begin{thebibliography}{99}
  

\setlength{\parsep }{-0.5ex}
\setlength{\itemsep}{-0.5ex}

\frenchspacing

%\newcommand\BAMS{\emph{Bull.\ Amer.\ Math.\ Soc.\ }}
%\newcommand\BIT{\emph{BIT\ }}
%\newcommand\Com{\emph{Computing\ }}
%\newcommand\CA{\emph{Constr.\ Approx.\ }}
%\newcommand\FCM{\emph{Found.\ Comput.\ Math.\ }}
%\newcommand\JAT{\emph{J.\ Approx.\ Th.\ }}
%\newcommand\JC{\emph{J.\ Complexity\ }}
%\newcommand\JMA{\emph{SIAM J.\ Math.\ Anal.\ }}
%\newcommand\JMAA{\emph{J.\ Math.\ Anal.\ Appl.\ }}
%\newcommand\JMM{\emph{J.\ Math.\ Mech.\ }}
%\newcommand\NM{\emph{Numer.\ Math.\ }}
%\newcommand\RMJ{\emph{Rocky Mt.\ J.\ Math.\ }}
%\newcommand\SJNA{\emph{SIAM J.\ Numer.\ Anal.\ }}
%\newcommand\SR{\emph{SIAM Rev.\ }}
%\newcommand\TAMS{\emph{Trans.\ Amer.\ Math.\ Soc.\ }}
%\newcommand\TOMS{\emph{ACM Trans.\ Math.\ Software\ }}
%\newcommand\USSR{\emph{USSR Comput.\ Maths.\ Math.\ Phys.\ }}

\frenchspacing

\addcontentsline{toc}{chapter}{Bibliography}

\bibitem{AMS96}
  \textsc{N.~Alon, Y.~Matias, M.~Szegedy.}
  \newblock The space complexity of approximating the frequency moments.
  \newblock \emph{STOC '96 Proceedings of the twenty-eighth annual ACM symposium on Theory of computing}, 
  20--29, 1996.
% J.\ Comput.\ System Sci, 58:137--147, 1999.

\bibitem{vBEs65}
  \textsc{B.~von Bahr, C.~Esseen}.
  \newblock Inequalities for the $r$th absolute moment
    of a sum of random variables, $1\leq r\leq2$.
  \newblock \emph{Ann.\ Math.\ Statist.}, 36:299--303, 1965.

\bibitem{Bakh59}
  \textsc{N.~S.~Bakhvalov}.
  \newblock On the approximate calculation of multiple integrals.
  \newblock {\em Vestnik MGU, Ser.\ Math.\ Mech.\ Astron.\ Phys.\ Chem.} 4:3--18,
  1959, in Russian.
  \newblock English translation: {\em J. Complexity}, 31(4):502--516, 2015.

\bibitem{BeSh88}
  \textsc{C.~Bennett, R.~Sharpley.}
  \newblock {\em Interpolation of operators},
    volume 129 of {\em Pure and Applied Mathematics}.
  \newblock Academic Press Inc., Boston, MA, 1988.

\bibitem{BC-BL13}
  \textsc{S.~Bubeck, N.~Cesa-Bianchi, G.~Lugosi}.
  \newblock Bandits with heavy tail.
  \newblock \emph{IEEE Transactions on Information Theory},
    59(11):7711--7717, 2013.

%\bibitem{Ca12}
%  \textsc{O.~Catoni.}
%  \newblock Challenging the empirical mean and empirical variance:
%    A deviation study.
%  \newblock {\em Annales de l’Institut Henri Poincar\'e
%    -- Probabilit\'es et Statistiques},
%    48(4):1148--1185, 2012.

\bibitem{CR65}
  \textsc{Y.S.~Chow, H.~Robbins}.
  \newblock {\em On the asymptotic theory
    of fixed-width sequential confidence intervals for the mean}.
  \newblock Annals of Mathematical Statistics, 36(2):457--462, 1965.

%\bibitem{CRS71}
%  \textsc{Y.S.~Chow, H.~Robbins, D.~Siegmund}.
%  \newblock {\em Great Expectations: The Theory of Optimal Stopping}.
%  \newblock Houghton Mifflin Company, 1971.

\checked{ 
\bibitem{CDHHZ14} 
\textsc{N.~Clancy, Y.~Ding, C.~Hamilton, F.J.~Hickernell and Y.~Zhang}.
 The cost of deterministic, adaptive, automatic algorithms: 
cones, not balls.
\emph{J.\ Complexity}, 30(1):21--45, 2014. 
} 

\bibitem{DKLR00}
  \textsc{P.~Dagum, R.~Karp, M.~Luby, S.~Ross.}
  \newblock An optimal algorithm for MC estimation.
  \newblock \emph{SIAM Journal on Computing}, 19(5):1484--1496, 2000.
%E    hier hat die Jahreszahl gefehlt!                          ^^^^ 

\bibitem{DLLO16}
  \textsc{L.~Devroye, M.~Lerasle, G.~Lugosi, R.I.~Oliveira.}
  \newblock Sub-gaussian mean estimators.
  \newblock Annals of Statistics, 44:2695--2725, 2016.

\bibitem{GNP13}
  \textsc{L.~Gajek, W.~Niemiro, P.~Pokarowski.}
  \newblock Optimal Monte Carlo integration with fixed relative precision.
  \newblock \emph{J.\ Complexity}, 29:4--26, 2013.

\bibitem{He94}
  \textsc{S.~Heinrich.}
  \newblock Random approximation in numerical analysis.
  \newblock \emph{Proceedings of the Conference
  ``Functional Analysis'' Essen}, pp.~123-171, 1994.

\checked{ 
\bibitem{H18} 
  \textsc{S.~Heinrich.}
\newblock On the complexity of computing the $L_q$ norm. 
\emph{J.\ Complexity}, in press, 2018. 
} 

\bibitem{HJLO13}
  \textsc{F.J.~Hickernell, L.~Jiang, Y.W.~Liu, A.~Owen}.
  \newblock Guaranteed conservative fixed width confidence
    intervals via Monte Carlo sampling.
    %In: %Dick J., Kuo F., Peters G., Sloan I.\ (eds)
    %\emph{Monte Carlo and Quasi-Monte Carlo Methods 2012}.
  \newblock %Springer Proceedings in Mathematics \& Statistics, vol 65,
    Proceedings of the MCQMC 2012,
    Springer, 2013.

%\bibitem{Hoe53}
%  \textsc{W.~Hoeffding}.
%  \newblock A lower bound for the average sample number of a sequential test.
%  \newblock \emph{Annals of Mathematical Statistics}, 24(1):127--130, 1953.

\bibitem{HS16}
  \textsc{D.~Hsu, S.~Sabato}.
  \newblock Loss minimization and parameter estimation with heavy tails.
  \newblock {\em Journal of Machine Learning Research}, 17:1--40, 2016.

\bibitem{Hu17Bernoulli}
 \textsc{M.~Huber}.
 \newblock A Bernoulli mean estimate with known relative error distribution.
 \newblock \emph{Random Structures and Algorithms}, 50:173--182, 2017.

%% Danke f\"ur den Hinweis, Daniel, ich habe die Referenz verbessert.
%\bibitem{Hu15}
%  \textsc{M.~Huber}.
%  \newblock An unbiased estimate for the mean of a $\{0,1\}$ random variable 
%    with relative error distribution independent of the mean.
%  \newblock Available on arXiv:1309.5413v2 [math.ST], 2015.

\bibitem{Hu17relVar}
  \textsc{M.~Huber}.
  \newblock An optimal $(\epsilon,\delta)$-approximation scheme
    for the mean of random variables with bounded relative variance.
  \newblock Available on arXiv:1706.01478v1 [stat.CO], 2017.

\bibitem{JVV86}
  \textsc{M.~Jerrum, L.~Valiant, V.~Vazirani.}
  \newblock Random generation of combinatorial structures
    from a uniform distribution.
  \newblock {\em Theoretical Computer Sciences}, 43:169--188, 1986.

\bibitem{JL16}
  \textsc{E.~Joly, G.~Lugosi}.
  \newblock Robust estimation of U-statistics.
  \newblock \emph{Stochastic Processes and their Applications},
    126(12):3760--3773.

\bibitem{JLO17}
  \textsc{E.~Joly, G.~Lugosi, R.I.~Oliveira}.
  \newblock On the estimation of the mean of a random vector.
  \newblock \emph{Electronic Journal of Statistics}, 11:440--451, 2017.

\checked{
\bibitem{KR18}
  \textsc{R.J.~Kunsch, D.~Rudolf}
  \newblock Optimal confidence for Monte Carlo integration of smooth functions.
  \newblock Available on arXiv:1809.09890 [math.NA], 2018.
}

\bibitem{LO12}
  \textsc{M.~Lerasle, R.I.~Oliveira}.
  \newblock Robust empirical mean estimators.
  \newblock Available on arXiv:1112.3914v1 [math.ST], 2012.

\bibitem{LM16}
  \textsc{G.~Lugosi, S.~Mendelson}.
  \newblock Risk minimization by median-of-means tournaments.
  \newblock Available on arXiv:1608.00757v1 [math.ST], 2016.

\bibitem{LM17}
  \textsc{G.~Lugosi, S.~Mendelson}.
  \newblock Regularization, sparse recovery, and median-of-means tournaments.
  \newblock Available on arXiv:1701.04112v1 [math.ST], 2017.
  
%\bibitem{MaVo}
%  \textsc{J.~Matou\v{s}ek, J.~Vondrak}.
%  \newblock \emph{The Probabilistic Method},
%  manuscript.

\bibitem{Mi15}
  \textsc{S.~Minsker.}
  \newblock Geometric median and robust estimation in Banach spaces.
  \newblock \emph{Bernoulli} 21(4):2308--2335, 2015.

\bibitem{NY83}
  \textsc{A.S.~Nemirovsky, D.B.~Yudin}.
  \newblock {\em Problem Complexity and Method Efficiency in Optimization.}
  \newblock John Wiley \& Sons, 1983.
  \newblock English translation, originally in Russian, 1979.

\bibitem{NiPo09}
  \textsc{W.~Niemiro, P.~Pokarowski}.
  \newblock Fixed precision MCMC estimation by median
    of product of averages.
  \newblock \emph{Journal of Applied Probability}, 46:309--329, 2009.

\bibitem{No88}
  \textsc{E.~Novak}.
  \newblock Deterministic and Stochastic Error Bounds in Numerical Analysis.
  \newblock \emph{Lecture Notes in Mathematics} 1349,
    Springer-Verlag, Berlin, 1988.

\bibitem{No16} 
  \textsc{E.~Novak}.
\newblock
 Some Results on the Complexity of Numerical Integration. 
\newblock 
In: Monte Carlo and Quasi-Monte Carlo Methods, MCQMC, Leuven, Belgium, 
Ronald Cools and Dirk Nuyens (eds.), Springer 2016.

\bibitem{NW08}
  \textsc{E.~Novak, H.~Wo\'zniakowski}.
  \newblock {\em Tractability of Multivariate Problems},  Volume~I,
    Linear Information.
  \newblock European Mathematical Society, 2008.% Z\"urich, 2008.

\bibitem{NW10}
  \textsc{E.~Novak, H.~Wo\'zniakowski}.
  \newblock {\em Tractability of Multivariate Problems},  Volume~II,
    Standard Information for Functionals.
  \newblock European Mathematical Society, 2010.

\bibitem{NW12}
  \textsc{E.~Novak, H.~Wo\'zniakowski}.
  \newblock {\em Tractability of Multivariate Problems},  Volume~III,
    Standard Information for Operators.
  \newblock European Mathematical Society, 2012.

\bibitem{ReLi01}
  \textsc{Y.~Ren, H.~Liang}
  \newblock On the best constant in Marcinkiewicz--Zygmund inequality.
  \newblock {\em Stat. Prob. Letters}, 53:227--233, 2001.  

\bibitem{RuSc15}
  \textsc{D.~Rudolf, N.~Schweizer}
  \newblock Error bounds of MCMC for functions with unbounded stationary variance.
  \newblock {\em Stat. Prob. Letters}, 99:6--12, 2015.  
  
\bibitem{Sieg85}
  \textsc{D.~Siegmund}.
  \newblock {\em Sequential Analysis: Tests and Confidence Intervals}.
  \newblock Springer-Verlag, 1985.

\bibitem{St45}
  \textsc{C.~Stein}
  \newblock A two-sample test for a linear hypothesis
    whose power is independent of the variance.
  \newblock {\em Annals of Mathematical Statistics}, 16:243--258, 1945.

\bibitem{TWW88}
  {\sc J.F.~Traub, G.W.~Wasilkowski, H.~Wo\'zniakowski.}
  \newblock {\em Information-Based Complexity}.
  \newblock Academic Press, 1988.

\bibitem{Wa45}
  {\sc A.~Wald}.
  \newblock Sequential tests of statistical hypotheses.
  \newblock {\em Annals of Mathematical Statistics}, 16:117--186, 1945.

\bibitem{Wa48}
  {\sc A.~Wald}.
  \newblock {\em Sequential Analysis}.
  \newblock John Wiley, 1948.

\end{thebibliography}
\end{document}